\documentclass{amsart}
\usepackage[utf8]{inputenc}
\usepackage{amsmath,amsthm,amssymb,amsfonts,stmaryrd,mathrsfs,dsfont}
\usepackage{color}

\theoremstyle{plain}
\newtheorem{theorem}{Theorem}[section]
\newtheorem{proposition}[theorem]{Proposition}
\newtheorem{corollary}[theorem]{Corollary}
\newtheorem{lemma}[theorem]{Lemma}

\theoremstyle{definition}

\theoremstyle{remark}

\DeclareMathOperator{\reg}{reg}
\DeclareMathOperator{\ureg}{ureg}

\DeclareMathOperator{\Pre}{Pre}

\DeclareMathOperator{\id}{id}
\DeclareMathOperator{\Sym}{Sym}

\begin{document}

\title[Semigroups of transformations whose characters belong to...]{Semigroups of transformations whose characters belong to a given semigroup}


\author[Mosarof Sarkar]{\bfseries Mosarof Sarkar}
\address{Department of Mathematics, Central University of South Bihar, Gaya--824236 (Bihar), India}
\email{mosarofsarkar@cusb.ac.in}
\author[Shubh N. Singh]{\bfseries Shubh N. Singh}
\address{Department of Mathematics, Central University of South Bihar, Gaya--824236 (Bihar), India}
\email{shubh@cub.ac.in}



\begin{abstract}
Let $X$ be a nonempty set and $\mathcal{P}=\{X_i\colon i\in I\}$ a partition of $X$. Denote by $T(X)$ the full transformation semigroup on $X$, and
$T(X, \mathcal{P})$ the subsemigroup of $T(X)$ consisting of all transformations that preserve $\mathcal{P}$. For every subsemigroup $\mathbb{S}(I)$ of $T(I)$, let $T_{\mathbb{S}(I)}(X,\mathcal{P})$ be the semigroup of all transformations $f\in T(X, \mathcal{P})$ such that $\chi^{(f)}\in \mathbb{S}(I)$, where $\chi^{(f)}\in T(I)$ defined by $i\chi^{(f)}=j$ whenever $X_if\subseteq X_j$. We describe regular and idempotent elements in $T_{\mathbb{S}(I)}(X,\mathcal{P})$, and determine when
$T_{\mathbb{S}(I)}(X,\mathcal{P})$ is a regular semigroup [inverse semigroup].  With the assumption that $\mathbb{S}(I)$ contains the identity, we characterize Green's relations on $T_{\mathbb{S}(I)}(X,\mathcal{P})$, describe unit-regular elements in $T_{\mathbb{S}(I)}(X,\mathcal{P})$, and determine when $T_{\mathbb{S}(I)}(X,\mathcal{P})$ is a unit-regular semigroup. We apply these general results to obtain more concrete results for $T(X,\mathcal{P})$.
\end{abstract}
\maketitle


\section{Introduction}
Throughout this paper, let $X$ be a nonempty set, let $\mathcal{P} = \{X_i\colon i\in I\}$ be a partition of $X$, and let $E$ be the equivalence relation on $X$ induced by $\mathcal{P}$. Denote by $T(X)$ the full transformation semigroup on $X$. The semigroup $T(X)$ and its subsemigroups have a crucial place in the theory of semigroups, since every semigroup is isomorphic to a subsemigroup of $T(Z)$ for some nonempty set $Z$ \cite[Theorem 1.1.2]{howie95}. In 1994, Pei \cite{pei-sf94} introduced the following subsemigroup of $T(X)$.

\begin{align*}
T(X, \mathcal{P}) &= \{f\in T(X)\colon \forall X_i \in \mathcal{P} ~\exists X_j \in \mathcal{P}, ~X_i f \subseteq X_j\}\\
&= \left\{f\in T(X)\colon \forall x, y\in X, ~(x,y)\in E \Longrightarrow (xf, yf)\in E\right\}.
\end{align*}
In that paper, Pei \cite[Theorem 2.8]{pei-sf94} proved that $T(X, \mathcal{P})$ is the semigroup of all continuous selfmaps on $X$ endowed with the topology having $\mathcal{P}$ as a basis. Since then the semigroup $T(X, \mathcal{P})$ has been extensively studied in the literature; see, for example, \cite{araujo-cps15, araujo-sf09,pei-sf96,pei-sf05, pei-ca05,sanwong-sf20,shubh-ca21,shubh-sf22, shubh-jaa22}.


\vspace{0.05cm}
Let $f\in T(X, \mathcal{P})$. The \emph{character} of $f$ is the transformation $\chi^{(f)}\colon I \to I$ defined by $i\chi^{(f)} = j$ whenever $X_i f \subseteq X_j$. If $X$ is finite, the character of $f\in T(X, \mathcal{P})$ has been studied using the symbol $\bar{f}$ by Ara\'{u}jo et al. \cite{araujo-cps15}, Dolinka and East \cite{east-ca16}, and Dolinka et al. \cite{east-bams16}. The character of $f\in T(X, \mathcal{P})$ for an arbitrary set $X$ was first studied by Purisang and Rakbud \cite{puri-ckms16}.
Let $\mathbb{S}(I)$ be a subsemigroup of $T(I)$. By using the notion of character, Rakbud \cite{rakbud-kmj18} introduced the following subsemigroup of $T(X, \mathcal{P})$:
\[T_{\mathbb{S}(I)}(X,\mathcal{P})=\left\{f\in T(X, \mathcal{P})\colon \chi^{(f)}\in \mathbb{S}(I)\right\}.\]
In \cite{rakbud-ijmm20}, Rakbud and Chaiya studied the semigroup $T_{\mathbb{S}(I)}(X,\mathcal{P})$. Semigroup $T_{\mathbb{S}(I)}(X,\mathcal{P})$ generalizes
the semigroup $T(X, \mathcal{P})$, since $T(X,\mathcal{P}) =T_{T(I)}(X,\mathcal{P})$. The semigroup $T_{\mathbb{S}(I)}(X,\mathcal{P})$ also generalizes both $T(X)$ and $\mathbb{S}(I)$ in the sense that $T_{\mathbb{S}(I)}(X,\mathcal{P})=T(X)$ if $|\mathcal{P}|= 1$, and $T_{\mathbb{S}(I)}(X,\mathcal{P}) \simeq
\mathbb{S}(I)$ if $\mathcal{P}$ consists of singleton sets. (Note that if $\mathcal{P}$ consists of singleton sets, then $|I|=|X|$ and $\mathbb{S}(I) \simeq
\mathbb{S}(X)$.)  For various proper subsemigroups $\mathbb{S}(I)$ of $T(I)$, the semigroup $T_{\mathbb{S}(I)}(X,\mathcal{P})$ has been investigated; see  \cite{deng-bims16,deng-sf10,deng-sf12,puri-ckms16, shubh-bkms23, shubh-ac23, shubh-ca21,shubh-sf22,shubh-jaa22, sun-jaa13}. 


\vspace{0.2mm}

The purpose of this paper is to study the semigroup $T_{\mathbb{S}(I)}(X,\mathcal{P})$, where $\mathbb{S}(I)$ is a general subsemigroup of $T(I)$. The paper is organized as follows. In the next section, we present the definitions and notation
used throughout this paper. In Section 3, we describe regular and idempotent elements in $T_{\mathbb{S}(I)}(X,\mathcal{P})$. We next determine when $T_{\mathbb{S}(I)}(X,\mathcal{P})$ is a regular semigroup, and when $T_{\mathbb{S}(I)}(X,\mathcal{P})$ is an inverse semigroup. Moreover, we prove that $T_{\mathbb{S}(I)}(X,\mathcal{P})$ is a regular semigroup if $\mathbb{S}(I)$ is a subgroup of the symmetric group on $I$, and thus generalize the result given by Purisang and Rakbud in \cite[Theorem 3.5(1)]{puri-ckms16}. With the assumption that $\mathbb{S}(I)$ contains the identity, we describe unit-regular elements in $T_{\mathbb{S}(I)}(X,\mathcal{P})$, and determine when $T_{\mathbb{S}(I)}(X,\mathcal{P})$ is a unit-regular semigroup in Section 4. In Section 5, we characterize Green's relations on $T_{\mathbb{S}(I)}(X,\mathcal{P})$ when $\mathbb{S}(I)$ contains the identity. Using these characterizations, we also obtain concrete and noticeably different descriptions of Green's relations on $T(X,\mathcal{P})$, whose Green's relations are known; see \cite{pei-ca05,pei-ca07}.


\section{Preliminaries and Notation}

Let $X$ be a nonempty set. The cardinality of $X$ is denoted by $|X|$.  For any nonempty set $Y$, we write $X\setminus Y = \{x\in X \colon x\notin Y\}$.
A \emph{partition} $\mathcal{P}$ of $X$ is a collection of nonempty disjoint subsets of $X$, called \emph{blocks} of $\mathcal{P}$, whose union is $X$. A \emph{trivial partition} is a partition that contains only singleton blocks or a single block. A \emph{transversal} of an equivalence relation $\rho$ on $X$ is a subset of $X$ that contains exactly one element from every $\rho$-class. We denote the identity mapping on $X$ by $\id_X$, and the symmetric group on $X$ by $\Sym(X)$.

\vspace{0.05cm}

Denote the composition of mappings by juxtaposition, and compose mappings from left to right. Let $f\colon X \to Y$ be a mapping. We write $xf$ for the image of $x$ under $f$. For a nonempty $A\subseteq X$  (resp. $B\subseteq Y$), we put $Af =\{af\colon a\in A\}$ (resp. $Bf^{-1}=\{x\in X\colon  xf \in B\}$). Furthermore, if $B=\{b\}$, then we write $bf^{-1}$ instead of $\{b\} f^{-1}$.
We let $D(f)= Y\setminus Xf$ and $\textnormal{d}(f)=|D(f)|$. The \emph{kernel} of $f$, denoted by $\ker(f)$, is an equivalence relation on $X$ defined by $\ker(f) = \{(a,b)\in X \times X \colon af = bf\}$. The symbol $\pi(f)$ denotes the partition of $X$ induced by $\ker(f)$, and $T_f$ denotes a transversal of $\ker(f)$. Note that $|X\setminus T_f|$ is independent of the choice of transversal of $\ker(f)$ (cf. \cite[p. 1356]{higgins-rsm98}). We let $\textnormal{c}(f) = |X\setminus T_f|$.


\vspace{0.05cm}

Let $f\in T(X)$. For a nonempty subset $A$ of the domain of $f$, the \emph{restriction} of $f$ to $A$ is the mapping $f_{\upharpoonright A}\colon A\to X$ defined by $x(f_{\upharpoonright A})=xf$ for all $x\in A$. Moreover if $B$ is a subset of the codomain of $f$ such that $Af \subseteq B$, then we use the same notation $f_{\upharpoonright A}$ for the mapping from $A$ to $B$ that assigns $xf$ to each $x\in A$. We say that $f$ \emph{preserves} or \emph{stabilizes} the partition $\mathcal{P}$ if for each $X_j \in \mathcal{P}$, there exists $X_k \in \mathcal{P}$ such that $X_j f\subseteq X_k$. Thus $T(X, \mathcal{P})$ is a subsemigroup of $T(X)$ consisting of all transformations that preserve the partition $\mathcal{P}$.

\vspace{0.05cm}
Let $S$ be a semigroup. An element $a\in S$ is an \textit{idempotent} if $a^2=a$. The set of all idempotents of $S$ is denoted by $E(S)$. It is well-known that $f\in T(X)$ is an idempotent if and only if $f$ acts as the identity map on its image set (cf. \cite[p. 6]{clifford-ams61}). An element $a\in S$ is \textit{regular} in $S$ if there is $b\in S$ such that $aba = a$. Such element $b$, if exists, is called an \emph{inner inverse} of $a$. The set of all inner inverses of $a$ is denoted by $\Pre(a)$, and the set of all regular elements in $S$ is denoted by $\reg(S)$. If $\reg(S) = S$, then we say that $S$ is \emph{regular}. The semigroup $S$ is an \textit{inverse} semigroup if for every $x\in S$, there is a unique $y\in S$ such that $xyx = x$ and $yxy = y$. Equivalently, the semigroup $S$ is an inverse semigroup if and only if $S$ is regular and its idempotents commute (cf. \cite[Theorem 5.1.1]{howie95}).

\vspace{0.05cm}
Let $S$ be a semigroup with identity. Denote by $U(S)$ the set of all units in $S$. Note that $U(T(X))=\Sym(X)$, and write $U\big(T(X,\mathcal{P})\big)=S(X,\mathcal{P})$. Clearly $S(X,\mathcal{P})=T(X,\mathcal{P})\cap \Sym(X)$. An element $a\in S$ is \textit{unit-regular} in $S$ if there is $u\in U(S)$ such that $aua = a$. The set of all unit-regular elements in $S$ is denoted by $\ureg(S)$. If $\ureg(S) = S$, then we say that $S$ is \emph{unit-regular}.

\vspace{0.05cm}

In what follows, the letter $I$ denotes the index set of the partition $\mathcal{P}$, and the symbol $\mathbb{S}(I)$ denotes a subsemigroup of $T(I)$. We refer the reader to \cite{howie95} for any undefined notation and terminology of semigroup theory.

\section{Regular and Inverse semigroups $T_{\mathbb{S}(I)}(X,\mathcal{P})$}

In this section, we give a characterization of regular elements in $T_{\mathbb{S}(I)}(X,\mathcal{P})$ and determine when $T_{\mathbb{S}(I)}(X,\mathcal{P})$ is a regular semigroup. Using these, we obtain known descriptions of regular elements in $T(X,\mathcal{P})$ and regularity of $T(X,\mathcal{P})$. If $\mathbb{S}(I)$ is a subgroup of $\Sym(I)$, we prove that $T_{\mathbb{S}(I)}(X,\mathcal{P})$ is regular, and thus generalize Theorem 3.5(1) in \cite{puri-ckms16}. Next, we give a characterization of idempotents in $T_{\mathbb{S}(I)}(X,\mathcal{P})$ and then determine when $T_{\mathbb{S}(I)}(X,\mathcal{P})$ is an inverse semigroup.

\vspace{0.1cm}
 We begin by recalling a specific case of Lemma 2.3 in \cite{puri-ckms16}.

\begin{lemma}\cite[Lemma 2.3]{puri-ckms16}\label{rak-lemma-2.3}
We have $\chi^{(fg)}=\chi^{(f)}\chi^{(g)}$ for all $f,g\in T(X,\mathcal{P})$.
\end{lemma}


In the next result, we observe that $\mathbb{S}(I)$ is a homomorphic image of $T_{\mathbb{S}(I)}(X,\mathcal{P})$.

\begin{proposition}\label{epimorphism}
The mapping $\varphi \colon T_{\mathbb{S}(I)}(X,\mathcal{P})\to \mathbb{S}(I)$ defined by $f\varphi=\chi^{(f)}$ is an epimorphism. Moreover, if $\id_X\in T_{\mathbb{S}(I)}(X,\mathcal{P})$, then $\id_X\varphi=\id_I$.
\end{proposition}

\begin{proof}[\textbf{Proof}]
Note from Lemma \ref{rak-lemma-2.3} that $\chi^{(fg)}=\chi^{(f)}\chi^{(g)}$ for all $f,g\in T_{\mathbb{S}(I)}(X,\mathcal{P})$. Then we see that $\varphi \colon T_{\mathbb{S}(I)}(X,\mathcal{P})\to \mathbb{S}(I)$ is a homomorphism, since
\[(fg)\varphi=\chi^{(fg)}=\chi^{(f)}\chi^{(g)}=(f\varphi)(g\varphi)\]
for all $f,g\in T_{\mathbb{S}(I)}(X,\mathcal{P})$. It remains to prove that $\varphi$ is surjective. For this, let $\alpha\in \mathbb{S}(I)$. Fix $x_i\in X_i$ for each $i\in I$, and define $h\in T(X)$ by $xh=x_{i\alpha}$ whenever $x\in X_i$ for some $i\in I$. Clearly $X_i h \subseteq X_{i\alpha}$, and so $\chi^{(h)}=\alpha$. This gives $h\in T_{\mathbb{S}(I)}(X,\mathcal{P})$. We also obtain $h\varphi =\chi^{(h)}=\alpha$ as required. If $\id_X\in T_{\mathbb{S}(I)}(X,\mathcal{P})$, then we get $\id_X\varphi=\chi^{(\id_X)}=\id_I$.	
\end{proof}


The following result gives a characterization of regular elements in $T_{\mathbb{S}(I)}(X,\mathcal{P})$.

\begin{theorem}\label{regular-element}
Let $f\in T_{\mathbb{S}(I)}(X,\mathcal{P})$. Then $f\in \reg\big(T_{\mathbb{S}(I)}(X,\mathcal{P})\big)$ if and only if there exists $\alpha\in \mathbb{S}(I)$ such that:
\begin{enumerate}
	\item[\rm(i)] $\chi^{(f)}\alpha\chi^{(f)}=\chi^{(f)}$;
	\item[\rm(ii)] $X_i\cap Xf\subseteq X_{i\alpha}f$ for all $i\in I\chi^{(f)}$. 
\end{enumerate}	
\end{theorem}

\begin{proof}[\textbf{Proof}]
Assume that $f\in \reg\big(T_{\mathbb{S}(I)}(X,\mathcal{P})\big)$. Then there exists $g\in T_{\mathbb{S}(I)}(X,\mathcal{P})$ such that $fgf=f$. This yields $\chi^{(f)}\chi^{(g)}\chi^{(f)}=\chi^{(f)}$ by Lemma \ref{rak-lemma-2.3}. Note that $\chi^{(g)}\in \mathbb{S}(I)$, and write $\chi^{(g)}=\alpha$. Then $\chi^{(f)}\alpha\chi^{(f)}=\chi^{(f)}$, and so \rm(i) holds.

\vspace{0.2mm}
To prove \rm(ii), let $i\in I\chi^{(f)}$ and $y\in X_i\cap Xf$. Then, since $y\in Xf$, there exists $x\in X$ such that $xf=y$. Also, since $y\in X_i$ and $\chi^{(g)} = \alpha$, we have $yg\in X_{i\alpha}$. Therefore, since $f = fgf$, we obtain $y=xf=(xf)gf=(yg)f\in X_{i\alpha}f$. Thus $X_i\cap Xf\subseteq X_jf$, and so \rm(ii) holds.

\vspace{0.2mm}
Conversely, assume that the given conditions hold. To prove $f\in \reg\big(T_{\mathbb{S}(I)}(X,\mathcal{P})\big)$, we will construct $g\in T_{\mathbb{S}(I)}(X,\mathcal{P})$ such that $fgf = f$. Let $i\in I\chi^{(f)}$. For each $x\in X_i\cap Xf$, by \rm(ii) there exists $x'\in X_{i\alpha}$ such that $x=x'f$. For each $j\in I$, we fix $x_j\in X_j$. Define $g\in T(X)$ by:
\begin{eqnarray*}
xg=
\begin{cases}
	x'  & \text{if $x\in X_i\cap Xf$, where $i\in I\chi^{(f)}$};\\
	x_{i\alpha} & \text{if $x\in X_i\setminus Xf$, where $i\in I\chi^{(f)}$};\\
	x_{j\alpha} & \text{if $x\in X_j$, where $j\in I\setminus I\chi^{(f)}$}.
\end{cases}
\end{eqnarray*}
Clearly $\chi^{(g)}=\alpha$. Since $\alpha \in \mathbb{S}(I)$, it follows that $g\in T_{\mathbb{S}(I)}(X,\mathcal{P})$. It is routine to verify that $fgf=f$, and hence $f\in \reg\big(T_{\mathbb{S}(I)}(X,\mathcal{P})\big)$.
\end{proof}		
	

Using \cite[Theorem 3.1]{magill-cjm74}, Pei \cite[Corollary 2.3]{pei-ca05} obtained the following result that characterizes regular elements in $T(X,\mathcal{P})$. Using Theorem \ref{regular-element}, we give an alternative proof of the same result.

\begin{proposition}
Let $f\in T(X,\mathcal{P})$. Then $f\in \reg\big(T(X,\mathcal{P})\big)$ if and only if for every $i\in I$, there exists $j\in I$ such that $X_i\cap Xf\subseteq X_jf$.
\end{proposition}

\begin{proof}[\textbf{Proof}]
Assume that $f\in \reg\big(T(X,\mathcal{P})\big)$, and let $i\in I$. If $i\in I\setminus I\chi^{(f)}$, then we get $X_i\cap Xf=\varnothing$, and so $X_i\cap Xf\subseteq X_jf$ for all $j\in I$. Assume that $i\in I\chi^{(f)}$. By Theorem \ref{regular-element}, since $T(X,\mathcal{P})=T_{T(I)}(X,\mathcal{P})$, there exists $\alpha \in T(I)$ such that $X_i\cap Xf\subseteq X_{i\alpha}f$. Thus we conclude that $X_i\cap Xf\subseteq X_jf$ for some $j\in I$.

\vspace{0.2mm}
Conversely, assume that the given condition holds. To prove $f\in \reg\big(T(X,\mathcal{P})\big)$, since $T(X,\mathcal{P}) = T_{T(I)}(X,\mathcal{P})$, we will construct $\alpha \in T(I)$ such that each condition of Theorem \ref{regular-element} holds. For each $i\in I$, by hypothesis we fix $j_i \in I$ such that $X_i\cap Xf\subseteq X_{j_i}f$. Consider $\alpha\in T(I)$ defined by $i\alpha=j_i$. We now show that $\chi^{(f)}\alpha\chi^{(f)}=\chi^{(f)}$. Let $k\in I$ and write $k\chi^{(f)}=\ell$. Then \[k\left(\chi^{(f)}\alpha\chi^{(f)}\right)= \left(k\chi^{(f)}\right)\alpha\chi^{(f)}=(\ell\alpha)\chi^{(f)}=j_{\ell}\chi^{(f)}=\ell=k\chi^{(f)},\]
and so $\chi^{(f)}\alpha\chi^{(f)}=\chi^{(f)}$. Thus the condition \rm(i) of Theorem \ref{regular-element} holds.

\vspace{0.2mm}
Next, let $i\in I\chi^{(f)}$. Then $i\alpha=j_i$ by the definition of $\alpha$. Therefore, by hypothesis and the choice of $j_i$, we get $X_i\cap Xf\subseteq X_{i\alpha}f$. Therefore the condition \rm(ii) of Theorem \ref{regular-element} holds.

\vspace{0.2mm}
Thus, since $T_{T(I)}(X,\mathcal{P}) = T(X,\mathcal{P})$, we conclude from Theorem \ref{regular-element} that $f\in \reg\big(T(X,\mathcal{P})\big)$.
\end{proof}

Note that the regularity of $\mathbb{S}(I)$ may not imply the regularity of $T_{\mathbb{S}(I)}(X,\mathcal{P})$ (cf. \cite[p. 631]{rakbud-kmj18}).
In the next result, we determine when $T_{\mathbb{S}(I)}(X,\mathcal{P})$ is a regular semigroup.

\begin{theorem}\label{regular-semigroup}
The semigroup $T_{\mathbb{S}(I)}(X,\mathcal{P})$ is regular if and only if:
\begin{enumerate}
	\item[\rm(i)] $\mathbb{S}(I)$ is regular;
	\item[\rm(ii)] for each $\alpha\in \mathbb{S}(I)$ if there exists $i\in I$ such that $|i\alpha^{-1}|\geq 2$, then $|X_i|=1$.
\end{enumerate}
\end{theorem}

\begin{proof}[\textbf{Proof}]
Assume that $T_{\mathbb{S}(I)}(X,\mathcal{P})$ is regular. Then \rm(i) is true by Proposition \ref{epimorphism} and the fact that any homomorphic image of a regular semigroup is regular (cf. \cite[Lemma 2.4.4]{howie95}).

\vspace{0.2mm}
To prove \rm(ii), let $\alpha \in \mathbb{S}(I)$. If $\mathcal{P}$ is trivial, then there is nothing to prove. Assume that $\mathcal{P}$ is nontrivial. Suppose to the contrary that $|X_j|\geq 2$ for some $j\in I$ with $|j\alpha^{-1}|\geq 2$. Choose distinct $y_1,y_2\in X_j$.
Fix $x_t\in X_t$ for each $t\in I\alpha\setminus \{j\}$ if $I\alpha\setminus \{j\}\neq \varnothing$. We also fix $\ell\in j\alpha^{-1}$. Define $f\in T(X)$ by:

\begin{eqnarray*}
xf=
\begin{cases}
	y_1         &\text{if $x\in X_{\ell}$};\\
	y_2         &\text{if $x\in X_k$, where $k\in j\alpha^{-1}\setminus\{\ell\}$};\\
	x_{i\alpha} &\text{if $x\in X_i$, where $i\in I\setminus j\alpha^{-1}$}.
\end{cases}
\end{eqnarray*}
Clearly $\chi^{(f)}=\alpha$. Since $\alpha \in \mathbb{S}(I)$, it follows that $f\in T_{\mathbb{S}(I)}(X,\mathcal{P})$. Therefore $f\in \reg \big(T_{\mathbb{S}(I)}(X,\mathcal{P})\big)$ by hypothesis. By \rm(i), let $\beta \in \Pre(\alpha)$. Note that $\alpha \beta \alpha = \alpha$ and $j \in I\alpha$, which gives $j\beta\in j\alpha^{-1}$. Therefore either $X_{j\beta}f=\{y_1\}$ or $X_{j\beta}f=\{y_2\}$ by the definition of $f$. However, we have $X_j\cap Xf=\{y_1,y_2\}$, and so $X_j\cap Xf\not\subseteq X_{j\beta}f$. This is a contradiction by Theorem \ref{regular-element}, since $f\in \reg \big(T_{\mathbb{S}(I)}(X,\mathcal{P})\big)$. Thus \rm(ii) holds.

\vspace{0.4mm}
Conversely, assume that the given conditions hold and let $f\in T_{\mathbb{S}(I)}(X,\mathcal{P})$. To prove the desired result, we will show by Theorem \ref{regular-element} that
$f\in \reg(T_{\mathbb{S}(I)}(X,\mathcal{P}))$. Since $f\in T_{\mathbb{S}(I)}(X,\mathcal{P})$, we have $\chi^{(f)}\in \mathbb{S}(I)$.
Then by \rm(i), there exists $\beta\in \mathbb{S}(I)$ such that $\chi^{(f)}\beta\chi^{(f)}=\chi^{(f)}$.

\vspace{0.2mm}
Next, let $i\in I\chi^{(f)}$. Then there are two cases:

\vspace{0.2mm}
\noindent\textbf{Case 1:} Suppose $|i\left(\chi^{(f)}\right)^{-1}|=1$. Write $i\left(\chi^{(f)}\right)^{-1}=\{j\}$. Clearly $X_i\cap Xf=X_jf$. Since $\chi^{(f)}\beta\chi^{(f)}=\chi^{(f)}$ and $j\chi^{(f)}=i$, we see that $i\beta=j$. Therefore $X_i\cap Xf\subseteq X_{i\beta}f$.

\vspace{0.2mm}
\noindent\textbf{Case 2:} Suppose $|i\left(\chi^{(f)}\right)^{-1}|\geq 2$. Then by $\rm(ii)$, we get $|X_i|=1$. Write $i\beta=k$. Then, since $\chi^{(f)}\beta\chi^{(f)}=\chi^{(f)}$, we see that $k\chi^{(f)}=i$. It follows that $X_kf=X_i$ and $X_i\cap Xf=X_i$. Therefore $X_i\cap Xf\subseteq X_{i\beta}f$.

\vspace{0.2mm}
Thus, in either case, we have $X_i\cap Xf\subseteq X_{i\beta}f$. Hence we conclude from Theorem \ref{regular-element} that  $f\in \reg(T_{\mathbb{S}(I)}(X,\mathcal{P}))$.
\end{proof}


If $\mathbb{S}(I) = \Sym(I)$, Purisang and Rakbud \cite[Theorem 3.5(1)]{puri-ckms16} proved that the semigroup $T_{\mathbb{S}(I)}(X,\mathcal{P})$ is regular. Using Theorem \ref{regular-semigroup}, we prove in the following result that $T_{\mathbb{S}(I)}(X,\mathcal{P})$ is regular for every subgroup $\mathbb{S}(I)$ of $\Sym(I)$, and thus
generalizes Theorem 3.5(1) in \cite{puri-ckms16}.

\begin{corollary}\label{group-SI-regular-TSIPX}
If $\mathbb{S}(I)$ is a subgroup of $\Sym(I)$, then $T_{\mathbb{S}(I)}(X,\mathcal{P})$ is regular.
\end{corollary}

\begin{proof}[\textbf{Proof}]
If $\mathbb{S}(I)$ is a subgroup of $\Sym(I)$, then it is obvious that $\mathbb{S}(I)$ is regular. Also for each $\alpha \in \mathbb{S}(I)$, we get $|i\alpha^{-1}|=1$ for all $i\in I$. Therefore the semigroup $T_{\mathbb{S}(I)}(X,\mathcal{P})$ is regular by Theorem \ref{regular-semigroup}.	
\end{proof}	


Using Theorem \ref{regular-semigroup}, we obtain the following result, which was first appeared in \cite[Proposition 2.4]{pei-ca05}.

\begin{proposition}\label{regular-TXP}
The semigroup $T(X,\mathcal{P})$ is regular if and only if $\mathcal{P}$ is trivial.
\end{proposition}

\begin{proof}[\textbf{Proof}]
Assume that $T(X,\mathcal{P})$ is regular. If $|X|\leq 2$, then there is nothing to prove. Assume that $|X|\geq 3$. Suppose to the contrary that $\mathcal{P}$ is not trivial. Then $|\mathcal{P}|\ge 2$ and $|X_j|\ge 2$ for some $j\in I$.
 Consider the constant mapping $\alpha\in T(I)$ defined by $i\alpha=j$. Since $|\mathcal{P}|\ge 2$, we get $|j\alpha^{-1}|\geq 2$. Therefore, since $T(X,\mathcal{P}) = T_{T(I)}(X,\mathcal{P})$, the semigroup $T(X,\mathcal{P})$ is not regular by Theorem \ref{regular-semigroup}, a contradiction. 

\vspace{0.2mm}
Conversely, assume that $\mathcal{P}$ is trivial. Then $T(X,\mathcal{P}) = T(X)$, and hence $T(X,\mathcal{P})$ is regular by \cite[Exercise 15, p. 63]{howie95}.
\end{proof}

Recall that a \emph{block mapping} is a
mapping whose both domain and codomain are blocks of a partition of a nonempty set (cf. \cite[Definition 5.1]{shubh-ca21}). Since $T_{\mathbb{S}(I)}(X,\mathcal{P})$ is a subsemigroup of $T(X,\mathcal{P})$, the following result is immediate from \cite[Lemma 5.2]{shubh-ca21}.

\begin{lemma}\label{block-lemma}
Let $f\in T(X)$. Then $f\in T_{\mathbb{S}(I)}(X,\mathcal{P})$ if and only if there exists a unique indexed family $B(f,I) = \{f_{\upharpoonright{X_i}}\colon i\in I\}$ of block mappings induced by $f$.
\end{lemma}

\vspace{0.1cm}
Note that an element $f\in T(X)$ is an idempotent if and only if $f$ acts as the identity on its image set (cf. \cite[p. 6]{clifford-ams61}). If $X$ is finite, the idempotents of $T(X,\mathcal{P})$ are characterized in \cite[Proposition 3.1]{east-ca16} and \cite[Proposition 3.1]{east-bams16} for a uniform partition $\mathcal{P}$ and a non-uniform partition $\mathcal{P}$, respectively.

\vspace{1.0mm}
The following result describes idempotents of $T_{\mathbb{S}(I)}(X,\mathcal{P})$.

\begin{proposition}\label{idmpotent}
Let $f\in T_{\mathbb{S}(I)}(X,\mathcal{P})$. Then $f\in E\big(T_{\mathbb{S}(I)}(X,\mathcal{P})\big)$ if and only if:
\begin{enumerate}
\item[\rm(i)] $\chi^{(f)}\in E(\mathbb{S}(I))$;
\item[\rm(ii)] $f_{\upharpoonright{X_i}}\in E(T(X_i))$ for all $i\in I\chi^{(f)}$;
\item[\rm(iii)] $X_if\subseteq X_{i\chi^{(f)}}f$ for all $i\in I\setminus I\chi^{(f)}$.
\end{enumerate}
\end{proposition}

\begin{proof}[\textbf{Proof}]
Assume that $f\in E\big(T_{\mathbb{S}(I)}(X,\mathcal{P})\big)$. Then $\chi^{(f)}\in \mathbb{S}(I)$ and $f^2=f$.
\begin{enumerate}
\item[\rm(i)] It is clear that $\chi^{(f)}\in E(\mathbb{S}(I))$, since
$\left(\chi^{(f)}\right)^2=\chi^{(f^2)}= \chi^{(f)}$ by Lemma \ref{rak-lemma-2.3}.

\vspace{0.1cm}	
\item[\rm(ii)] Let $i\in I\chi^{(f)}$, and let $y\in X_if_{\upharpoonright{X_i}}$. Then $i\chi^{(f)}=i$ by $\rm(i)$, and $y\in Xf$. Therefore, since $f$ is idempotent, we obtain $yf_{\upharpoonright{X_i}}=yf=y$. Hence $f_{\upharpoonright{X_i}}\in E(T(X_i))$.

\vspace{0.1cm}	
\item[\rm(iii)] It is straightforward, since $f$ stabilizes $\mathcal{P}$ and $f^2 = f$.
\end{enumerate}	

\vspace{0.2mm}	
Conversely, assume that the given conditions hold. To prove $f\in E\big(T_{\mathbb{S}(I)}(X,\mathcal{P})\big)$, let $y\in Xf$. Then $y\in X_i$ for some $i\in I\chi^{(f)}$. Therefore $i\chi^{(f)}=i$ by \rm(i), and $f_{\upharpoonright{X_i}}\in E(T(X_i))$ by \rm(ii). By \rm(iii), we then see that $y\in X_if_{\upharpoonright{X_i}}$. Thus we obtain $yf=yf_{\upharpoonright{X_i}}=y$. Hence $f\in E\big(T_{\mathbb{S}(I)}(X,\mathcal{P})\big)$.
\end{proof}


Recall that a semigroup is an inverse semigroup if and only if it is regular and its idempotents commute (cf. \cite[Theorem 5.1.1]{howie95}).

\vspace{0.1cm}
The following result determines when $T_{\mathbb{S}(I)}(X,\mathcal{P})$ is an inverse semigroup.

\begin{theorem}\label{inverse-semigroup}
The semigroup $T_{\mathbb{S}(I)}(X,\mathcal{P})$ is an inverse semigroup if and only if:
\begin{enumerate}
\item[\rm(i)] $\mathbb{S}(I)$ is an inverse semigroup;
\item[\rm(ii)] for each $\alpha\in E(\mathbb{S}(I))$, we have $|X_i|=1$ for all $i\in I\alpha$.
\end{enumerate}
\end{theorem}

\begin{proof}[\textbf{Proof}]
Assume that $T_{\mathbb{S}(I)}(X,\mathcal{P})$ is an inverse semigroup. Then \rm(i) is true by Proposition \ref{epimorphism} and the fact that any homomorphic image of an inverse semigroup is an inverse semigroup (cf. \cite[Theorem 5.1.4]{howie95}).

\vspace{0.5mm}
To prove \rm(ii), let $\alpha\in E(\mathbb{S}(I))$. If $|X|= 1$, then there is nothing to prove. Assume that $|X|\geq 2$.  Suppose to the contrary that there exists $j\in I\alpha$ such that $|X_j|\geq 2$. Since $\alpha\in E(\mathbb{S}(I))$ and $j\in I\alpha$, we get $j\alpha=j$. Now, choose distinct $y, z \in X_j$. If $I\setminus\{j\}\neq \varnothing$, then we fix $x_i\in X_i$ for each $i\in I\setminus\{j\}$. Define $f,g\in T(X)$ by:

\begin{eqnarray*}
xf=
\begin{cases}
	y           & \text{if $x\in X_i$, where $i\in j\alpha^{-1}$};\\
	x_{i\alpha} &  \text{if $x\in X_i$, where $i\in I\setminus j\alpha^{-1}$}.
\end{cases}
\end{eqnarray*}
and
\begin{eqnarray*}
xg=
\begin{cases}
	z  & \text{if $x\in X_i$, where $i\in j\alpha^{-1}$};\\
	x_{i\alpha} &  \text{if $x\in X_i$, where $i\in I\setminus j\alpha^{-1}$}.
\end{cases}
\end{eqnarray*}
Clearly $\chi^{(f)}=\alpha=\chi^{(g)}$. Since $\alpha \in E(\mathbb{S}(I))$, it follows that $\chi^{(f)}, \chi^{(g)}\in E(\mathbb{S}(I))$. Now, we observe from the definition of $f$ that $f_{\upharpoonright{X_i}}\in E(T(X_i))$ for all $i\in I\alpha$, and $X_if_{\upharpoonright{X_i}}=X_{i\alpha}f_{\upharpoonright{X_{i\alpha}}}$ for all $i\in I\setminus I\alpha$. Therefore $f\in E\big(T_{\mathbb{S}(I)}(X,\mathcal{P})\big)$ by Proposition \ref{idmpotent}. Similarly, we can obtain $g\in E\big(T_{\mathbb{S}(I)}(X,\mathcal{P})\big)$. However, since $y(fg)\neq y(gf)$, we have $fg\neq gf$. This contradicts our assumption that $T_{\mathbb{S}(I)}(X,\mathcal{P})$ is an inverse semigroup. Thus $|X_i|=1$ for all $i\in I\alpha$, and hence \rm(ii) holds.	

\vspace{0.2mm}
Conversely, assume that the given conditions hold. To prove that $T_{\mathbb{S}(I)}(X,\mathcal{P})$ is an inverse semigroup, we will show that
$T_{\mathbb{S}(I)}(X,\mathcal{P})$ is regular and its idempotents commute.

\vspace{0.2mm}
First, we show that $T_{\mathbb{S}(I)}(X,\mathcal{P})$ is regular. By \rm(i), it is clear that $\mathbb{S}(I)$ is regular. Now, let $\alpha\in \mathbb{S}(I)$. Since $\mathbb{S}(I)$ is regular, there exists $\beta\in E(\mathbb{S}(I))$ such that $I\alpha=I\beta$ (cf. \cite[Theorem 2.3.2 and Exercise 16(a) on p. 63]{howie95}). Then by \rm(ii), we have $|X_i|=1$ for all $i\in I\beta$. It follows that $|X_i|=1$ for all $i\in I\alpha$. Thus $T_{\mathbb{S}(I)}(X,\mathcal{P})$ is regular by Theorem \ref{regular-semigroup}.

\vspace{0.2mm}
To prove that idempotents of $T_{\mathbb{S}(I)}(X,\mathcal{P})$ commute, let $f,g\in E\big(T_{\mathbb{S}(I)}(X,\mathcal{P})\big)$ and let $x\in X$. Then $\chi^{(f)},\chi^{(g)}\in E(\mathbb{S}(I))$ by Proposition \ref{idmpotent}, and $x\in X_i$ for some $i\in I$. Write $i\chi^{(fg)}=j$. Observe from \rm(i) and Lemma \ref{rak-lemma-2.3} that $i\chi^{(gf)}=i(\chi^{(g)}\chi^{(f)})= i(\chi^{(f)}\chi^{(g)})= i\chi^{(fg)}=j$. Note also from Lemma \ref{rak-lemma-2.3} that $I\chi^{(fg)} = I(\chi^{(f)}\chi^{(g)})\subseteq I\chi^{(g)}$, and so $j\in I\chi^{(g)}$. By \rm(ii), we then get $X_j=\{y\}$ for some $y\in X$. Therefore we obtain $x(fg)=y=x(gf)$. Hence idempotents of $T_{\mathbb{S}(I)}(X,\mathcal{P})$ commute.

\vspace{0.2mm}
Thus we conclude that $T_{\mathbb{S}(I)}(X,\mathcal{P})$ is an inverse semigroup.

\end{proof}		

\section{Unit-regularity of $T_{\mathbb{S}(I)}(X,\mathcal{P})$} 
Throughout this section, assume that $\mathbb{S}(I)$ contains the identity $\id_I$. We describe unit-regular elements in $T_{\mathbb{S}(I)}(X,\mathcal{P})$. Then we determine when $T_{\mathbb{S}(I)}(X,\mathcal{P})$ is a unit-regular semigroup. Using this, we determine when $T(X,\mathcal{P})$ is a unit-regular semigroup.

\vspace{0.2mm}
Observe that $\mathbb{S}(I)$ contains the identity $\id_I$ if and only if $T_{\mathbb{S}(I)}(X,\mathcal{P})$ contains the identity $\id_X$. Note that  $U\big(T_{\mathbb{S}(I)}(X,\mathcal{P})\big)\subseteq U\big(T(X,\mathcal{P})\big)=S(X,\mathcal{P})$, and so $U\big(T_{\mathbb{S}(I)}(X,\mathcal{P})\big)=T_{\mathbb{S}(I)}(X,\mathcal{P})\cap S(X,\mathcal{P})$.

\vspace{1.0mm}
We require the next three results from \cite{shubh-ca21, shubh-jaa22} to prove Theorem 
\ref{uregular-element}, which describes unit-regular elements in $T_{\mathbb{S}(I)}(X,\mathcal{P})$.

\begin{lemma}\cite[Lemma 3.6\rm(i)]{shubh-ca21}\label{lem-bij-image-block}
If $f\in S(X, \mathcal{P})$, then $X_i f\in \mathcal{P}$ for all $i\in I$.
\end{lemma}

\begin{lemma}\cite[Theorem 5.8]{shubh-ca21}\label{tm-char-of-ele-of-S(X, P)}
Let $f\in T(X, \mathcal{P})$. Then $f\in S(X, \mathcal{P})$ if and only if:
\begin{enumerate}
\item[\rm(i)] every element of $B(f, I)$ is bijective;
\item[\rm(ii)] $\chi^{(f)}$ is bijective.
\end{enumerate}
\end{lemma}

\begin{lemma}\cite[Lemma 3.1]{shubh-jaa22}\label{lm-transveral}
Let $f\colon X\to Y$ and $g\colon Y\to X$ be mappings. If $fgf = f$, then $X(fg)$ is a transversal of the equivalence relation $\ker(f)$.
\end{lemma}


\begin{theorem}\label{uregular-element}
Let $f\in T_{\mathbb{S}(I)}(X,\mathcal{P})$. Then $f\in \ureg\big(T_{\mathbb{S}(I)}(X,\mathcal{P})\big)$ if and only if there exists $\alpha\in U(\mathbb{S}(I))$ such that:
\begin{enumerate}
\item[\rm(i)] $\chi^{(f)}\alpha\chi^{(f)}=\chi^{(f)}$;
\item[\rm(ii)] $X_i\cap Xf\subseteq X_{i\alpha}f$ for all $i\in I\chi^{(f)}$; 
\item[\rm(iii)] $|X_i|=|X_{i\alpha}|$ for all $i\in I$;
\item[\rm(iv)] $\textnormal{c}(f_{\upharpoonright{X_{i\alpha}}})=\textnormal{d}(f_{\upharpoonright{X_{i\alpha}}})$
 for all $i\in I\chi^{(f)}$.
\end{enumerate}
\end{theorem}

\begin{proof}[\textbf{Proof}]
Assume that $f\in \ureg\big(T_{\mathbb{S}(I)}(X,\mathcal{P})\big)$. Then there exists $g\in U\big(T_{\mathbb{S}(I)}(X,\mathcal{P})\big)$ such that $fgf=f$. This gives $\chi^{(f)}\chi^{(g)}\chi^{(f)}=\chi^{(f)}$ by Lemma \ref{rak-lemma-2.3}. Note that $\chi^{(g)}\in U(\mathbb{S}(I))$, and write $\chi^{(g)}=\alpha$.

\begin{enumerate}
\item[\rm(i)] It is clear from the above.
	
\vspace{0.2mm}	
\item[\rm(ii)]
Let $i\in I\chi^{(f)}$ and $y\in X_i\cap Xf$. Then, since $y\in Xf$, there exists $x\in X$ such that $xf=y$. Also, since $y\in X_i$ and $\chi^{(g)} = \alpha$, we see that $yg\in X_{i\alpha}$. Therefore, since $f= fgf$, we obtain $y=xf=(xf)gf=(yg)f\in X_{i\alpha}f$. Hence $X_i\cap Xf\subseteq X_{i\alpha}f$.

\vspace{0.2mm}	
\item[\rm(iii)]	
Let $i\in I$. Note that $g\in U\big(T_{\mathbb{S}(I)}(X,\mathcal{P})\big)$ and $\chi^{(g)} = \alpha$. Then, since  $U\big(T_{\mathbb{S}(I)}(X,\mathcal{P})\big)\subseteq S(X,\mathcal{P})$, we get $X_ig=X_{i\alpha}$ by Lemma \ref{lem-bij-image-block}, which implies $|X_i|=|X_{i\alpha}|$.

\vspace{0.2mm}	
\item[\rm(iv)]	
Let $i\in I\chi^{(f)}$, and write $i\alpha=j$. Recall that $g\in U\big(T_{\mathbb{S}(I)}(X,\mathcal{P})\big)$ and $\chi^{(g)} = \alpha$. Since $U\big(T_{\mathbb{S}(I)}(X,\mathcal{P})\big)\subseteq S(X,\mathcal{P})$, it follows from Lemma \ref{lem-bij-image-block} that $X_ig=X_j$. This yields $X_ig_{\upharpoonright{X_i}}=X_j$. By \rm(ii), we have $X_jf_{\upharpoonright{X_j}}\subseteq X_i$. Then, since $fgf=f$, we obtain $f_{\upharpoonright{X_j}}g_{\upharpoonright{X_i}} f_{\upharpoonright{X_j}}= f_{\upharpoonright{X_j}}$. It follows that $X_j(f_{\upharpoonright{X_j}}g_{\upharpoonright{X_i}})$ is a transversal of $\ker(f_{\upharpoonright{X_j}})$ by Lemma \ref{lm-transveral}. Write $T_{f_{\upharpoonright{X_j}}} = X_j(f_{\upharpoonright{X_j}}g_{\upharpoonright{X_i}})$. Since $g\in S(X,\mathcal{P})$, it follows that $g_{\upharpoonright{X_i}}$ is bijective by Theorem \ref{tm-char-of-ele-of-S(X, P)}. Therefore $(X_i\setminus X_jf_{\upharpoonright{X_j}})g_{\upharpoonright{X_i}}=X_j\setminus T_{f_{\upharpoonright{X_j}}}$, and hence $\textnormal{d}(f_{\upharpoonright{X_j}})=\textnormal{c}(f_{\upharpoonright{X_j}})$.
\end{enumerate}
Conversely, assume that the given conditions hold. To prove $f\in \ureg\big(T_{\mathbb{S}(I)}(X,\mathcal{P})\big)$, we will construct $g\in U(T_{\mathbb{S}(I)}(X,\mathcal{P}))$ such that $fgf = f$. Let $i\in I\chi^{(f)}$, and write $i\alpha=j$. Choose a transversal $T_{f_{\upharpoonright{X_j}}}$ of $\ker(f_{\upharpoonright{X_j}})$, and define a mapping $h_i'\colon X_jf_{\upharpoonright{X_j}}\to T_{f_{\upharpoonright{X_j}}}$ by $xh_i'=a_x$ if and only if $x(f_{\upharpoonright{X_j}})^{-1}\cap T_{f_{\upharpoonright{X_j}}}=\{a_x\}$. It is easy to verify that $h_i'$ is a bijection. By \rm(iv), there exists a bijection $h_i''\colon D(f_{\upharpoonright{X_j}})\to C(f_{\upharpoonright{X_j}})$, where $C(f_{\upharpoonright{X_j}})=X_j\setminus T_{f_{\upharpoonright{X_j}}}$. By \rm(i), we see that $j\chi^{(f)} = i$. Thus we may write $X_i=X_jf_{\upharpoonright{X_j}}\cup D(f_{\upharpoonright{X_j}})$ and $X_j=T_{f_{\upharpoonright{X_j}}}\cup C(f_{\upharpoonright{X_j}})$. Define a mapping $h_i\colon X_i\to X_j$ by:
\begin{eqnarray*}
xh_i=
\begin{cases}
xh_i'   & \text{if $x\in X_jf_{\upharpoonright{X_j}}$};\\
xh_i''  & \text{if $x\in D(f_{\upharpoonright{X_j}})$}.
\end{cases}
\end{eqnarray*}
Clearly $h_i$ is a bijection. Now for each $j\in I\setminus I\chi^{(f)}$, by \rm(iii) there is a bijection $g_j\colon X_j\to X_{j\alpha}$. Finally, define $g\in T(X)$ by:
\begin{eqnarray*}
xg=
\begin{cases}
	xh_i   & \text{if $x\in X_i$, where $i\in I\chi^{(f)}$};\\
	xg_j  & \text{if $x\in X_j$, where $j\in I\setminus I\chi^{(f)}$}.
\end{cases}
\end{eqnarray*}
Clearly $\chi^{(g)}=\alpha$. Since $\alpha \in U(\mathbb{S}(I))$, it follows that $g$ is a bijection. Therefore we have $g\in U\big(T_{\mathbb{S}(I)}(X,\mathcal{P})\big)$. It is routine to verify that $fgf=f$, and hence $f\in \ureg\big(T_{\mathbb{S}(I)}(X,\mathcal{P})\big)$.
\end{proof}	

We require the following result to prove Theorem \ref{uregular-semigroup}, which determines unit-regularity of $T_{\mathbb{S}(I)}(X,\mathcal{P})$.

\begin{lemma}\label{cf-not-df}
Let $X$ and $Y$ be nonempty sets. Then there is a mapping $f\colon X \to Y$ such that $\textnormal{c}(f)\neq \textnormal{d}(f)$ if any of the following conditions hold:
\begin{enumerate}
	\item[\rm(i)] $|X|\neq |Y|$;
	\item[\rm(ii)] $X$ is infinite;
	\item[\rm(iii)] $Y$ is infinite.
\end{enumerate}	
\end{lemma}

\begin{proof}[\textbf{Proof}]\
\begin{enumerate}
\item[\rm(i)] If $|X|<|Y|$, then there exists an injection $f\colon X\to Y$. Therefore $\textnormal{c}(f)=0$ and $|Xf| = |X|$. Since $|X|<|Y|$, we have $|Xf|<|Y|$. Therefore $\textnormal{d}(f)\neq 0$, and hence $\textnormal{c}(f)\neq \textnormal{d}(f)$. If $|X|>|Y|$, we may dually prove that $\textnormal{c}(f)\neq \textnormal{d}(f)$.	

\vspace{0.2mm}
\item[\rm(ii)] If $|X|\neq |Y|$, then we are done by \rm(i). Assume that $|X|=|Y|$. Since $X$ is infinite, there exists an injection $f\colon X \to Y$ that is not a surjection. Therefore $\textnormal{c}(f)=0$ but $\textnormal{d}(f)\neq 0$. Hence $\textnormal{c}(f)\neq \textnormal{d}(f)$.
 
\vspace{0.2mm}
\item[\rm(iii)] It is similar to \rm(ii).
\end{enumerate}
\end{proof}	

\begin{theorem}\label{uregular-semigroup}
The semigroup $T_{\mathbb{S}(I)}(X,\mathcal{P})$ is unit-regular if and only if:
\begin{enumerate}
	\item[\rm(i)] $\mathbb{S}(I)$ is unit-regular;
	\item[\rm(ii)] for each $\alpha\in U(\mathbb{S}(I))$, we have $|X_i|=|X_{i\alpha}|$ for all $i\in I$;
	\item[\rm(iii)] $|X_i|$ is finite for all $i\in I$;
	\item[\rm(iv)] for each $\beta\in \mathbb{S}(I)$ if there exists $i\in I$ such that $|i\beta^{-1}|\geq 2$, then $|X_i|=1$.
\end{enumerate}		
\end{theorem}

\begin{proof}[\textbf{Proof}]
Assume that $T_{\mathbb{S}(I)}(X,\mathcal{P})$ is unit-regular.

\begin{enumerate}
\item[\rm(i)]
The semigroup $\mathbb{S}(I)$ is unit-regular by Proposition \ref{epimorphism} and the fact that any homomorphic image of a unit-regular semigroup is unit-regular (cf. \cite[Proposition 2.7]{hickey-prse97}).

\vspace{0.05cm}	
\item[\rm(ii)]
Let $\alpha \in U(\mathbb{S}(I))$. Then $\alpha \in \Sym(I)$. Suppose to the contrary that there exists $j\in I$ such that $|X_j|\neq |X_{j\alpha}|$. Then by Lemma \ref{cf-not-df}, there exists a mapping $h\colon X_j\to X_{j\alpha}$ such that $\textnormal{c}(h)\neq \textnormal{d}(h)$. For each $i\in I\setminus\{j\alpha\}$, we fix $x_i\in X_i$. Define $f\in T(X)$ by:
\begin{eqnarray*}
xf=
\begin{cases}
	xh       & \text{if $x\in X_j$};\\
	x_{i\alpha} & \text{if $x\in X_i$, where $i\in I\setminus \{j\}$}.
\end{cases}	
\end{eqnarray*}
Clearly $\chi^{(f)}=\alpha$, and so $f\in T_{\mathbb{S}(I)}(X,\mathcal{P})$. Now, let $j\alpha=k$ and $\alpha^{-1}=\beta$. Clearly $\beta\in  U(\mathbb{S}(I))$ and $k\beta=j$. Note that $k\in I\chi^{(f)}$ and $f_{\upharpoonright{X_{k\beta}}} = f_{\upharpoonright{X_{j}}} = h$. Since $\textnormal{c}(h)\neq \textnormal{d}(h)$, it follows that $\textnormal{c}(f_{\upharpoonright{X_{k\beta}}})\neq \textnormal{d}(f_{\upharpoonright{X_{k\beta}}})$. This means $f$ does not satisfy the condition $\rm(iv)$ of Theorem \ref{uregular-element}, and so $f\notin \ureg\big(T_{\mathbb{S}(I)}(X,\mathcal{P})\big)$ by Theorem \ref{uregular-element}. This is a contradiction, because $T_{\mathbb{S}(I)}(X,\mathcal{P})$ is unit-regular. Hence $|X_i|=|X_{i\alpha}|$ for all $i\in I$.

\vspace{0.2mm}	
\item[\rm(iii)]
Suppose to the contrary that there exists $j\in I$ such that $|X_j|$ is infinite. Take $\alpha \in U(\mathbb{S}(I))$, and write $j\alpha=\ell$. Since $|X_j|$ is infinite, by Lemma \ref{cf-not-df} there is a mapping $h\colon X_j\to X_{\ell}$ such that $\textnormal{c}(h)\neq \textnormal{d}(h)$. For each $i\in I\setminus\{j\alpha\}$, we fix $x_i\in X_i$. Define $f\in T(X)$ by:
\begin{eqnarray*}
xf=
\begin{cases}
	xh       & \text{if $x\in X_j$};\\
	x_{i\alpha} & \text{if $x\in X_i$, where $i\in I\setminus \{j\}$}.
\end{cases}	
\end{eqnarray*}
Clearly $\chi^{(f)}=\alpha$, and so $f\in T_{\mathbb{S}(I)}(X,\mathcal{P})$. Write $\alpha^{-1}=\beta$. Clearly $\beta\in  U(\mathbb{S}(I))$ and $\ell\beta=j$. Note that $\ell\in I\chi^{(f)}$ and  $f_{\upharpoonright{X_{\ell\beta}}} = f_{\upharpoonright{X_{j}}} = h$. Since $\textnormal{c}(h)\neq \textnormal{d}(h)$, it follows that $\textnormal{c}(f_{\upharpoonright{X_{\ell\beta}}})\neq \textnormal{d}(f_{\upharpoonright{X_{\ell\beta}}})$. This means $f$ does not satisfy the condition $\rm(iv)$ of Theorem \ref{uregular-element}, and so $f\notin \ureg\big(T_{\mathbb{S}(I)}(X,\mathcal{P})\big)$ by Theorem \ref{uregular-element}. This is a contradiction, because $T_{\mathbb{S}(I)}(X,\mathcal{P})$ is unit-regular. Hence $|X_i|$ is finite for all $i\in I$.

\vspace{0.2mm}		
\item[\rm(iv)]
It is directly followed from Theorem \ref{regular-semigroup}, since $T_{\mathbb{S}(I)}(X,\mathcal{P})$ is a regular semigroup.
\end{enumerate}

Conversely, assume that the given conditions hold, and let $f\in T_{\mathbb{S}(I)}(X,\mathcal{P})$. To prove the desired result, we will show that $f\in \ureg(T_{\mathbb{S}(I)}(X,\mathcal{P}))$ by Theorem \ref{uregular-element}. Since $f\in T_{\mathbb{S}(I)}(X,\mathcal{P})$, we have $\chi^{(f)}\in \mathbb{S}(I)$. By \rm(i), there exists $\alpha\in U(\mathbb{S}(I))$ such that $\chi^{(f)}\alpha\chi^{(f)}=\chi^{(f)}$. By \rm(ii), we have $|X_i|=|X_{i\alpha}|$ for all $i\in I$.

\vspace{0.2mm}
Next, let $i\in I\chi^{(f)}$. Since $\chi^{(f)}\alpha\chi^{(f)}=\chi^{(f)}$, we see that $X_{i\alpha}f \subseteq X_i$. That means $f_{\upharpoonright{X_{i\alpha}}}$ is a block mapping from $X_{i\alpha}$ to $X_i$. Note from $\rm(ii)$ and \rm(iii) that $|X_i|=|X_{i\alpha}|$ and $X_i$ is finite, respectively. Therefore $\textnormal{c}(f_{\upharpoonright{X_{i\alpha}}})=\textnormal{d}(f_{\upharpoonright{X_{i\alpha}}})$.

\vspace{0.2mm}
Finally, since $i\in I\chi^{(f)}$, there are two cases:

\vspace{0.2mm}

\noindent\textbf{Case 1:} Suppose $|i(\chi^{(f)})^{-1}|=1$. Since $\chi^{(f)}\alpha\chi^{(f)}=\chi^{(f)}$, we see that $(i\alpha)\chi^{(f)}=i$ and so $i\left(\chi^{(f)}\right)^{-1}=\{i\alpha\}$. This gives $X_{i\alpha}f=X_i\cap Xf$, and so $ X_i\cap Xf\subseteq X_{i\alpha}f$.

\vspace{0.2mm}
\noindent\textbf{Case 2:} Suppose $|i(\chi^{(f)})^{-1}|\geq 2$. Then by \rm(iv), we get $|X_i|=1$. Therefore, since $i\in I\chi^{(f)}$, we obtain $X_i\cap Xf=X_i$. Since $\chi^{(f)}\alpha\chi^{(f)}=\chi^{(f)}$, we see that $(i\alpha)\chi^{(f)}=i$. Therefore $X_{i\alpha}f=X_i$, and so $X_i\cap Xf\subseteq X_{i\alpha}f$.

\vspace{0.05cm}
In either case, we have $X_i\cap Xf\subseteq X_{i\alpha}f$. Thus $f$ satisfies each condition of Theorem \ref{uregular-element}, and hence $f\in \ureg(T_{\mathbb{S}(I)}(X,\mathcal{P}))$ by Theorem \ref{uregular-element}.
\end{proof}	

The authors \cite[Theorem 4.7]{shubh-jaa22} recently described unit-regular elements in $T(X,\mathcal{P})$. We require the following result to prove Proposition \ref{uregular-TXP}, which determines unit-regularity of $T(X,\mathcal{P})$.

\begin{proposition}\cite[Proposition 5]{d'alarco-sf80}\label{pr-unit-reg-of-Tx}
The semigroup $T(X)$ is unit-regular if and only if $X$ is finite. 
\end{proposition}

\begin{proposition}\label{uregular-TXP}
The semigroup $T(X,\mathcal{P})$ is unit-regular if and only if $X$ is finite and $\mathcal{P}$ is trivial.
\end{proposition}

\begin{proof}[\textbf{Proof}]
Assume that $T(X,\mathcal{P})$ is unit-regular. Then, since $T(X,\mathcal{P})=T_{T(I)}(X,\mathcal{P})$, the semigroup $T(I)$ is unit-regular by Theorem \ref{uregular-semigroup}(i). Therefore $I$ is finite by Proposition \ref{pr-unit-reg-of-Tx}. It follows from Theorem \ref{uregular-semigroup}\rm(iii) that $X$ is finite. Next, since $T(X,\mathcal{P})$ is regular, the partition $\mathcal{P}$ is trivial by Proposition \ref{regular-TXP}.
	
\vspace{0.2mm}
Conversely, assume that the given conditions hold. Since $\mathcal{P}$ is trivial, it follows that $T(X,\mathcal{P})=T(X)$. Moreover, since $X$ is finite, the semigroup $T(X,\mathcal{P})$ is unit-regular by Proposition \ref{pr-unit-reg-of-Tx}.
\end{proof}	

\section{Green's Relations on $T_{\mathbb{S}(I)}(X,\mathcal{P})$} 
Throughout this section, assume that the semigroup $\mathbb{S}(I)$ contains the identity $\id_I$. We characterize Green's relations on $T_{\mathbb{S}(I)}(X,\mathcal{P})$. Using these characterizations, we obtain more concrete and noticeably different descriptions of Green's relations on $T(X,\mathcal{P})$, whose Green's relations are known; see \cite{pei-ca05,pei-ca07}. We refer the reader to \cite[Chapter 2]{howie95} for the basic definitions and notation related to Green's relations on a semigroup.


\vspace{0.05cm}
First, we recall the following well-known characterizations of Green's relations $\mathcal{L}$, $\mathcal{R}$, $\mathcal{D}$, and $\mathcal{J}$ on $T(X)$ (cf. \cite[p. 52]{clifford-ams61}, \cite[p. 63, Exercise 16]{howie95}).
\begin{theorem}\label{green-Tx}
Let $f,g\in T(X)$. Then:
\begin{enumerate}
	\item [(1)] $L_f\leq L_g$ if and only if $Xf\subseteq Xg$. Subsequently, $(f,g)\in \mathcal{L}$ if and only if $Xf=Xg$;
	\item [(2)] $R_f\leq R_g$ if and only if $\ker(g)\subseteq \ker(f)$. Subsequently, $(f,g)\in \mathcal{R}$ if and only if $\ker(f)=\ker(g)$;
	\item [(3)] $J_f\leq J_g$ if and only if $|Xf|\leq |Xg|$. Subsequently, $(f,g)\in \mathcal{J}$ if and only if $|Xf|=|Xg|$;
	\item [(4)] $\mathcal{D} = \mathcal{J}$.
\end{enumerate}	
\end{theorem}

\vspace{0.05cm}
We require the following result to obtain a characterization of the relation $\mathcal{L}$ on $T_{\mathbb{S}(I)}(X,\mathcal{P})$.

\begin{lemma}\label{l-lemma}
Let $f,g\in T_{\mathbb{S}(I)}(X,\mathcal{P})$. Then $L_f\leq L_g$ in $T_{\mathbb{S}(I)}(X,\mathcal{P})$ if and only if there exists $\alpha\in \mathbb{S}(I)$ such that $\chi^{(f)}=\alpha\chi^{(g)}$ and $X_if\subseteq X_{i\alpha}g$ for all $i\in I$.
\end{lemma}

\begin{proof}[\textbf{Proof}]
Assume that $L_f\leq L_g$ in $T_{\mathbb{S}(I)}(X,\mathcal{P})$. Then $f=hg$ for some $h\in T_{\mathbb{S}(I)}(X,\mathcal{P})$. This gives $\chi^{(f)}=\chi^{(hg)}=\chi^{(h)}\chi^{(g)}$ by Lemma \ref{rak-lemma-2.3}. Write $\chi^{(h)}=\alpha$, and thus $\chi^{(f)}=\alpha\chi^{(g)}$. To prove the remaining part, let $i\in I$. Notice that $X_ih\subseteq X_{i\alpha}$, and therefore $X_if=X_i(hg) =(X_ih)g\subseteq X_{i\alpha}g$.

\vspace{0.2mm}
Conversely, assume that the given conditions hold. To prove $L_f\leq L_g$ in $T_{\mathbb{S}(I)}(X,\mathcal{P})$, we will construct $h\in T_{\mathbb{S}(I)}(X,\mathcal{P})$ such that $f = hg$. For each $i\in I$ and for each $x\in X_i$, by hypothesis we fix $y\in X_{i\alpha}$ such that $xf=yg$. Define $h\in T(X)$ by $xh=y$ whenever $x\in X_i$ for some $i\in I$. It is clear by the definition of $h$ that $\chi^{(h)}=\alpha$, and so $h\in T_{\mathbb{S}(I)}(X,\mathcal{P})$. To prove $f=hg$, let $x\in X$. Then $x\in X_i$ for some $i\in I$. Therefore $x(hg)=(xh)g=yg=xf$, which yields $f=hg$. Hence $L_f\leq L_g$ in $T_{\mathbb{S}(I)}(X,\mathcal{P})$.
\end{proof}	

\begin{theorem}\label{l-green}
Let $f,g\in T_{\mathbb{S}(I)}(X,\mathcal{P})$. Then $(f,g)\in \mathcal{L}$ in $T_{\mathbb{S}(I)}(X,\mathcal{P})$ if and only if there exist $\alpha,\beta\in \mathbb{S}(I)$ such that:
\begin{enumerate}
	\item[\rm(i)] $\chi^{(f)}=\alpha\chi^{(g)}$ and $\chi^{(g)}=\beta\chi^{(f)}$;
	\item[\rm(ii)] $X_if\subseteq X_{i\alpha}g$ and $X_ig\subseteq X_{i\beta}f$ for all $i\in I$.
\end{enumerate}
\end{theorem}

\begin{proof}[\textbf{Proof}]
It follows directly from Lemma \ref{l-lemma}.	
\end{proof}


In 2005, Pei \cite[Theorem 3.2]{pei-ca05} gave a characterization of the relation $\mathcal{L}$ on $T(X,\mathcal{P})$. Using Theorem \ref{l-green}, we also provide another characterization of the relation $\mathcal{L}$ on $T(X,\mathcal{P})$ as follows.

\begin{proposition}
Let $f,g\in T(X,\mathcal{P})$. Then $(f,g)\in \mathcal{L}$ in $T(X,\mathcal{P})$ if and only if for each $i\in I$, there exist $j,k\in I$ such that $X_if\subseteq X_jg$ and $X_ig\subseteq X_kf$.
\end{proposition}

\begin{proof}[\textbf{Proof}]
Assume that $(f,g)\in \mathcal{L}$ in $T(X,\mathcal{P})$, and let $i\in I$. By Theorem \ref{l-green}, since $T(X,\mathcal{P}) = T_{T(I)}(X,\mathcal{P})$, there exist $\alpha,\beta\in T(I)$ such that $X_if\subseteq X_{i\alpha}g$ and $X_ig\subseteq X_{i\beta}f$. Write $i\alpha=j$ and $i\beta=k$. Thus $X_if\subseteq X_jg$ and $X_ig\subseteq X_kf$.

\vspace{0.1cm}
Conversely, assume that the given conditions hold. For each $i\in I$, by hypothesis we choose $j_i,k_i\in I$ such that $X_if\subseteq X_{j_i}g$ and $X_ig\subseteq X_{k_i}f$. Define $\alpha, \beta\in T(I)$ by $i\alpha=j_i$ and $i\beta=k_i$. Then we have $X_if\subseteq X_{i\alpha}g$ and $X_ig\subseteq X_{i\beta}f$ for all $i\in I$. Next, we prove that $\chi^{(f)}=\alpha\chi^{(g)}$. For this, let $i \in I$. Then $X_if\subseteq X_{j_i}g$, and so $i\chi^{(f)} = j_i\chi^{(g)}$. Therefore we obtain $i\left(\alpha\chi^{(g)}\right)=(i\alpha)\chi^{(g)}=j_i\chi^{(g)}=i\chi^{(f)}$, which yields $\chi^{(f)}=\alpha\chi^{(g)}$. Similarly, we may prove that $\chi^{(g)}=\beta\chi^{(f)}$ by using $X_ig\subseteq X_{k_i}f$.
Thus, since $T_{T(I)}(X,\mathcal{P}) = T(X,\mathcal{P})$, we conclude from Theorem \ref{l-green} that $(f,g)\in \mathcal{L}$ in $T(X,\mathcal{P})$.	
\end{proof}	


To obtain a characterization of the relation $\mathcal{R}$ on $T_{\mathbb{S}(I)}(X,\mathcal{P})$, we require the following definition and Lemma \ref{r-lemma}.

\vspace{0.1cm}
Let $f,g\in T(X)$. We say that $\pi(f)$ \textit{refines} $\pi(g)$, denoted by $\pi(f)\preceq \pi(g)$, if $\ker(f)\subseteq \ker(g)$. We write $\pi(f)=\pi(g)$ if $\pi(f)\preceq \pi(g)$ and $\pi(g)\preceq \pi(f)$.

\begin{lemma}\label{r-lemma}
Let $f,g\in T_{\mathbb{S}(I)}(X,\mathcal{P})$. Then $R_f\leq R_g$ in $T_{\mathbb{S}(I)}(X,\mathcal{P})$ if and only if $R_{\chi^{(f)}}\leq R_{\chi^{(g)}}$ in $\mathbb{S}(I)$ and $\pi(g)\preceq \pi(f)$.
\end{lemma}

\begin{proof}[\textbf{Proof}]
Assume that $R_f\leq R_g$ in $T_{\mathbb{S}(I)}(X,\mathcal{P})$. Then $f=gh$ for some $h\in T_{\mathbb{S}(I)}(X,\mathcal{P})$. This gives $\chi^{(f)}=\chi^{(gh)}=\chi^{(g)}\chi^{(h)}$ by Lemma \ref{rak-lemma-2.3}. Since $\chi^{(f)},\chi^{(g)},\chi^{(h)}\in \mathbb{S}(I)$, it follows that $R_{\chi^{(f)}}\leq R_{\chi^{(g)}}$ in $\mathbb{S}(I)$. Next, since $T_{\mathbb{S}(I)}(X,\mathcal{P})\subseteq T(X)$, we have $f,g,h\in T(X)$. Therefore $\ker(g)\subseteq \ker(f)$ by Theorem \ref{green-Tx}(2), and hence $\pi(g)\preceq \pi(f)$.

\vspace{0.1cm}
Conversely, assume that the given conditions hold. To prove $R_f\leq R_g$ in $T_{\mathbb{S}(I)}(X,\mathcal{P})$, we will construct $h \in T_{\mathbb{S}(I)}(X,\mathcal{P})$ such that $f = gh$. Since $R_{\chi^{(f)}}\leq R_{\chi^{(g)}}$ in $\mathbb{S}(I)$, we have $\chi^{(f)}=\chi^{(g)}\beta$ for some $\beta\in \mathbb{S}(I)$. Now, for each $i\in I$, fix $x_i\in X_i$; for each $x\in Xg$, fix $x'\in xg^{-1}$. Define $h\in T(X)$ by:
\begin{align*}
xh=
\begin{cases}
	x'f        & \text{if $x\in X_i\cap Xg$};\\
	x_{i\beta} &  \text{if $x\in X_i\setminus Xg$.}
\end{cases}
\end{align*}
It is clear that $h$ is well-defined, since $\pi(g)\preceq \pi(f)$. Since $\chi^{(f)}=\chi^{(g)}\beta$, we see that $\chi^{(h)}=\beta$, and so $h\in T_{\mathbb{S}(I)}(X,\mathcal{P})$. To prove $f=gh$, let $x\in X$. Then $x\in X_i$ for some $i\in I$. Therefore, since $x'\in xg^{-1}$ and $xg^{-1}\in \pi(f)$, we obtain $x(gh)=(xg)h=x'f=xf$, which yields $f=gh$. Hence $R_f\leq R_g$ in $T_{\mathbb{S}(I)}(X,\mathcal{P})$.
\end{proof}	

\begin{theorem}\label{r-green}
Let $f,g\in T_{\mathbb{S}(I)}(X,\mathcal{P})$. Then $(f,g)\in \mathcal{R}$ in $T_{\mathbb{S}(I)}(X,\mathcal{P})$ if and only if $\left(\chi^{(f)},\chi^{(g)}\right)\in \mathcal{R}$ in $\mathbb{S}(I)$ and $\pi(f)=\pi(g)$.
\end{theorem}

\begin{proof}[\textbf{Proof}]
It follows directly from Lemma \ref{r-lemma}.	
\end{proof}

Recall that $T(X,\mathcal{P})=T_{T(I)}(X,\mathcal{P})$. Using Theorem \ref{r-green} and Theorem \ref{green-Tx}(2), we immediately have the following characterization of the relation $\mathcal{R}$ on $T(X,\mathcal{P})$, which was first appeared in \cite[Theorem 3.1]{pei-ca05}.

\begin{corollary}
Let $f,g\in T(X,\mathcal{P})$. Then $(f,g)\in \mathcal{R}$ in $T(X,\mathcal{P})$ if and only if $\pi\left(\chi^{(f)}\right)=\pi\left(\chi^{(g)}\right)$ and $\pi(f)=\pi(g)$.
\end{corollary}


\vspace{0.1cm}
We require the following result that is useful in establishing the relation  $\mathcal{D}$ on $T_{\mathbb{S}(I)}(X,\mathcal{P})$.
\begin{lemma}\label{k-bg-k-sm}
Let $f,g\in T_{\mathbb{S}(I)}(X,\mathcal{P})$.
\begin{enumerate}
	\item[\rm(i)] If $L_f\leq L_g$ in $T_{\mathbb{S}(I)}(X,\mathcal{P})$, then $L_{\chi^{(f)}}\leq L_{\chi^{(g)}}$ in $\mathbb{S}(I)$. Subsequently, if $(f,g)\in \mathcal{L}$ in $T_{\mathbb{S}(I)}(X,\mathcal{P})$, then $(\chi^{(f)},\chi^{(g)})\in \mathcal{L}$ in $\mathbb{S}(I)$.
	\vspace{0.05cm}
	\item[\rm(ii)] If $R_f\leq R_g$ in $T_{\mathbb{S}(I)}(X,\mathcal{P})$, then $R_{\chi^{(f)}}\leq R_{\chi^{(g)}}$ in $\mathbb{S}(I)$. Subsequently, if $(f,g)\in \mathcal{R}$ in $T_{\mathbb{S}(I)}(X,\mathcal{P})$, then $(\chi^{(f)},\chi^{(g)})\in \mathcal{R}$ in $\mathbb{S}(I)$.
\end{enumerate}
\end{lemma}

\begin{proof}[\textbf{Proof}]\
\begin{enumerate}
	\item[\rm(i)] If $L_f\leq L_g$ in $T_{\mathbb{S}(I)}(X,\mathcal{P})$, then $f=hg$ for some $h\in T_{\mathbb{S}(I)}(X,\mathcal{P})$. This gives $\chi^{(f)}=\chi^{(hg)} =\chi^{(h)}\chi^{(g)}$ by Lemma \ref{rak-lemma-2.3}. Since $\chi^{(f)},\chi^{(g)},\chi^{(h)}\in \mathbb{S}(I)$, it follows that $L_{\chi^{(f)}}\leq L_{\chi^{(g)}}$ in $\mathbb{S}(I)$. The rest of the proof is immediate from the previous statements.
	
	\vspace{0.05cm}
	\item[\rm(ii)] It is the dual of the proof above.
\end{enumerate}
\end{proof}		
For every $A\subseteq X$ and every $f\in T(X)$, let $\pi_A(f)=\{M\in \pi(f)\colon M\cap A\neq \varnothing\}$. If $A=X$, then we write $\pi(f)$ instead of $\pi_X(f)$.

\vspace{1.0mm}
The following result characterizes the relation $\mathcal{D}$ on $T_{\mathbb{S}(I)}(X,\mathcal{P})$.
\begin{theorem}\label{d-green}
Let $f,g\in T_{\mathbb{S}(I)}(X,\mathcal{P})$. Then $(f,g)\in \mathcal{D}$ in $T_{\mathbb{S}(I)}(X,\mathcal{P})$ if and only if there exist $\alpha,\beta,\gamma\in \mathbb{S}(I)$ and a bijection $\varphi\colon \pi(f)\to \pi(g)$ such that:
\begin{enumerate}
	\item[\rm(i)] $\chi^{(f)}=\alpha\gamma$, $\gamma=\beta\chi^{(f)}$, and $\left(\gamma,\chi^{(g)}\right)\in \mathcal{R}$ in $\mathbb{S}(I)$;
	
	\item[\rm(ii)] $\pi_{X_i}(f)\varphi\subseteq \pi_{X_{i\alpha}}(g)$  for all $i\in I$;
	
	\item[\rm(iii)] $\pi_{X_i}(g)\varphi^{-1}\subseteq \pi_{X_{i\beta}}(f)$ for all $i\in I$.
\end{enumerate}
\end{theorem}

\begin{proof}[\textbf{Proof}]
Assume that $(f,g)\in \mathcal{D}$ in $T_{\mathbb{S}(I)}(X,\mathcal{P})$. Then, since $\mathcal{D}=\mathcal{L}\circ \mathcal{R}$, there exists $h\in T_{\mathbb{S}(I)}(X,\mathcal{P})$ such that $(f,h)\in \mathcal{L}$ and $(h,g)\in \mathcal{R}$ in $T_{\mathbb{S}(I)}(X,\mathcal{P})$. Write $\chi^{(h)}=\gamma$. Clearly $\gamma \in \mathbb{S}(I)$. Since $(f,h)\in \mathcal{L}$ in $T_{\mathbb{S}(I)}(X,\mathcal{P})$, by Theorem \ref{l-green} there exist $\alpha,\beta\in \mathbb{S}(I)$ such that $\chi^{(f)}=\alpha\gamma$ and $\gamma=\beta\chi^{(f)}$; $X_if\subseteq X_{i\alpha}h$ and $X_ih\subseteq X_{i\beta}f$ for all $i\in I$. Also, since $(h,g)\in \mathcal{R}$ in $T_{\mathbb{S}(I)}(X,\mathcal{P})$, we get $(\gamma,\chi^{(g)})\in \mathcal{R}$ in $\mathbb{S}(I)$ by Lemma \ref{k-bg-k-sm}(ii). Thus $\rm(i)$ holds.

\vspace{0.05cm}
To prove (ii) and (iii), we first define a mapping $\varphi\colon \pi(f)\to \pi(g)$ by $(af^{-1})\varphi=ah^{-1}$ whenever $a\in Xf$. Since $(f,h)\in \mathcal{L}$ in $T_{\mathbb{S}(I)}(X,\mathcal{P})$, we have $(f,h)\in \mathcal{L}$ in $T(X)$, and so $Xf=Xh$ by Theorem \ref{green-Tx}(1). Since $a\in Xf$, it follows that $a\in Xh$, and so $ah^{-1}\in \pi(h)$. Also, since $(h,g)\in \mathcal{R}$ in $T_{\mathbb{S}(I)}(X,\mathcal{P})$, we have $(h,g)\in \mathcal{R}$ in $T(X)$, and so $\pi(h)=\pi(g)$ by Theorem \ref{green-Tx}(2). Therefore, since $ah^{-1}\in \pi(h)$, we obtain $ah^{-1}\in \pi(g)$. Thus $\varphi$ is well-defined. Moreover, since $Xf=Xh$ and $\pi(h)=\pi(g)$, it is a routine matter to verify that $\varphi$ is bijective.
	
\vspace{0.05cm}
Now	to prove $\rm(ii)$, let $i\in I$ and let $A\in \pi_{X_i}(f)$. Then $A\cap X_i\neq \varnothing$. Write $Af=\{a\}$. Since $A\cap X_i\neq \varnothing$ and $f$ stabilizes $\mathcal{P}$, it is obvious that $a\in X_if$. Note that $X_if\subseteq X_{i\alpha}h$, and so $a\in X_{i\alpha}h$. This gives $ah^{-1}\cap X_{i\alpha}\neq \varnothing$. Therefore $A\varphi \cap X_{i\alpha}=(af^{-1})\varphi\cap X_{i\alpha}=ah^{-1}\cap X_{i\alpha}\neq \varnothing$, which yields $A\varphi\in \pi_{X_{i\alpha}}(g)$. Hence $\pi_{X_i}(f)\varphi\subseteq \pi_{X_{i\alpha}}(g)$, and thus $\rm(ii)$ holds.
	
\vspace{0.05cm}
Next to prove $\rm(iii)$, let $i\in I$ and let $B\in \pi_{X_i}(g)$. Then  $B\cap X_i\neq \varnothing$. Since $\pi(h)=\pi(g)$, we get $\pi_{X_i}(h)=\pi_{X_i}(g)$. It follows that $B\in \pi_{X_i}(h)$. Write $Bh=\{b\}$. Clearly $b\in X_ih$. Note that $X_ih\subseteq X_{i\beta}f$, and so $b\in X_{i\beta}f$. This gives $bf^{-1}\cap X_{i\beta}\neq \varnothing$. Therefore $B\varphi^{-1}\cap X_{i\beta}= (bh^{-1})\varphi^{-1}\cap X_{i\beta}=bf^{-1}\cap X_{i\beta}\neq \varnothing$, which yields $B\varphi^{-1}\in \pi_{X_{i\beta}}(f)$. Hence $\pi_{X_i}(g)\varphi^{-1}\subseteq \pi_{X_{i\beta}}(f)$, and thus $\rm(iii)$ holds.

\vspace{0.1cm}
Conversely, assume that the given conditions hold. To prove $(f,g)\in \mathcal{D}$ in $T_{\mathbb{S}(I)}(X,\mathcal{P})$, we will construct $h\in T_{\mathbb{S}(I)}(X,\mathcal{P})$ such that $(f,h)\in \mathcal{L}$ and $(h,g)\in \mathcal{R}$ in $T_{\mathbb{S}(I)}(X,\mathcal{P})$. Define $h\in T(X)$ by $xh=y$ whenever $\big((xgg^{-1})\varphi^{-1}\big)f=\{y\}$. First, we show that $\chi^{(h)}=\gamma$. For this, let $i\in I$ and let $x\in X_i$. Clearly $xgg^{-1}\cap X_i\neq \varnothing$, and so $xgg^{-1}\in \pi_{X_i}(g)$. This gives $(xgg^{-1})\varphi^{-1} \in \pi_{X_{i\beta}}(f)$ by $\rm(iii)$, and so $(xgg^{-1})\varphi^{-1}\cap X_{i\beta}\neq \varnothing$. By (i), since $\gamma=\beta\chi^{(f)}$, we have $(i\beta)\chi^{(f)}=i\gamma$, and so $\big((xgg^{-1})\varphi^{-1}\big)f = \{y\}\subseteq X_{i\gamma}$. By the definition of $h$, we therefore have $xh=y\in X_{i\gamma}$. This gives $X_ih\subseteq X_{i\gamma}$, and so $i\chi^{(h)}=i\gamma$. Hence $\chi^{(h)}=\gamma$, and thus $h\in T_{\mathbb{S}(I)}(X,\mathcal{P})$.

\vspace{0.05cm}
It is a routine matter to verify that $\pi(h)=\pi(g)$. By (i), we also get $\left(\chi^{(h)},\chi^{(g)}\right)\in \mathcal{R}$ in $\mathbb{S}(I)$. Therefore  $(h,g)\in \mathcal{R}$ in $T_{\mathbb{S}(I)}(X,\mathcal{P})$ by Theorem \ref{r-green}.

\vspace{0.05cm}
We only require to show that $(f,h)\in \mathcal{L}$ in $T_{\mathbb{S}(I)}(X,\mathcal{P})$. Note that $\chi^{(h)}=\gamma$. By $\rm(i)$, we have $\chi^{(f)}=\alpha\chi^{(h)}$ and $\chi^{(h)}=\beta\chi^{(f)}$. Therefore by Theorem \ref{l-green}, it remains to prove that $X_if\subseteq X_{i\alpha}h$ and $X_ih\subseteq X_{i\beta}f$ for all $i\in I$. For this, let $i\in I$. To prove $X_if\subseteq X_{i\alpha}h$, let $y\in X_if$. Then $yf^{-1}\cap X_i\neq \varnothing$, and so $yf^{-1}\in \pi_{X_i}(f)$. Therefore by $\rm(ii)$, we have $(yf^{-1})\varphi\in \pi_{X_{i\alpha}}(g)$. This yields $(yf^{-1})\varphi\cap X_{i\alpha}\neq \varnothing$. Choose $z\in (yf^{-1})\varphi\cap X_{i\alpha}$. Then $zgg^{-1} \cap X_{i\alpha}\neq \varnothing$, and so $zgg^{-1}\in\pi_{X_{i\alpha}}(g)$.
Therefore, since $z\in zgg^{-1}$ and $z\in (yf^{-1})\varphi$, we obtain $(yf^{-1})\varphi=zgg^{-1}$. This gives $\big((zgg^{-1})\varphi^{-1}\big)f=\{y\}$. By the definition of $h$, we get $zh=y$. Hence $y\in X_{i\alpha}h$, and thus $X_if\subseteq X_{i\alpha}h$.

\vspace{0.05cm}
To prove $X_ih\subseteq X_{i\beta}f$, let $y\in X_ih$. Then $yh^{-1}\cap X_i\neq \varnothing$, and so $yh^{-1}\in \pi_{X_i}(h)$. Since $\pi(h)=\pi(g)$, it follows easily that $\pi_{X_i}(h)=\pi_{X_i}(g)$. Therefore $yh^{-1}\in \pi_{X_i}(g)$. By (iii), we then have $(yh^{-1})\varphi^{-1}\in \pi_{X_{i\beta}}(f)$, and so $(yh^{-1})\varphi^{-1}\cap X_{i\beta}\neq \varnothing$. Choose $z\in (yh^{-1})\varphi^{-1}\cap X_{i\beta}$. Then $zff^{-1} \cap X_{i\beta}\neq \varnothing$, and so $zff^{-1}\in\pi_{X_{i\beta}}(f)$. Therefore, since $z\in zff^{-1}$ and $z\in (yh^{-1})\varphi^{-1}$, we obtain $(yh^{-1})\varphi^{-1}=zff^{-1}$. This gives $\big((yh^{-1})\varphi^{-1}\big)f=\{zf\}$. By the definition of $h$, we get $zf=y$. Hence $y\in X_{i\beta}f$, and thus $X_ih\subseteq X_{i\beta}f$. Thus $(f,h)\in \mathcal{L}$ in $T_{\mathbb{S}(I)}(X,\mathcal{P})$ by Theorem \ref{l-green}. Hence, since $\mathcal{L}\circ \mathcal{R}= \mathcal{D}$, we conclude that $(f,g)\in \mathcal{D}$ in $T_{\mathbb{S}(I)}(X,\mathcal{P})$.
\end{proof}
\vspace{0.1cm}
In 2005, Pei \cite[Theorem 3.4]{pei-ca05} determined the relation $\mathcal{D}$ on $T(X,\mathcal{P})$. Sangkhanan and Sanwong \cite[Corollary 3.10]{sanwong-sf20} also obtained a different description of $\mathcal{D}$ on $T(X,\mathcal{P})$. Using Theorem \ref{d-green}, we obtain a new characterization of $\mathcal{D}$ on $T(X,\mathcal{P})$ as follows.

\begin{proposition}
Let $f,g\in T(X,\mathcal{P})$. Then $(f,g)\in \mathcal{D}$ in $T(X,\mathcal{P})$ if and only if there exist $\gamma\in T(I)$ and a bijection $\varphi\colon \pi(f)\to \pi(g)$ such that:
\begin{enumerate}
	\item[\rm(i)] $\left(\chi^{(f)},\gamma\right)\in \mathcal{L}$ and $\left(\gamma,\chi^{(g)}\right)\in \mathcal{R}$ in $T(I)$;
	\item[\rm(ii)] for each $i\in I$, there exist $j\in \left(i\chi^{(f)}\right)\gamma^{-1}$ and $k\in (i\gamma)\left(\chi^{(f)}\right)^{-1}$ such that $\pi_{X_i}(f)\varphi\subseteq \pi_{X_j}(g)$ and $\pi_{X_i}(g)\varphi^{-1}\subseteq \pi_{X_k}(f)$.
\end{enumerate}
\end{proposition}

\begin{proof}[\textbf{Proof}]
Assume that $(f,g)\in \mathcal{D}$ in $T(X,\mathcal{P})$. By Theorem \ref{d-green}, since $T(X,\mathcal{P}) = T_{T(I)}(X,\mathcal{P})$, there exist $\alpha,\beta,\gamma\in T(I)$ and a bijection $\varphi\colon \pi(f)\to \pi(g)$ such that each condition of Theorem \ref{d-green} holds. Thus \rm(i) follows immediately from Theorem \ref{d-green}\rm(i).
	
\vspace{0.05cm}
To prove \rm(ii), let $i\in I$. Write $i\alpha=j$ and $i\beta=k$. First, we show that $j\in \left(i\chi^{(f)}\right)\gamma^{-1}$. By Theorem \ref{d-green}$\rm(i)$, we have $\chi^{(f)}=\alpha\gamma$. Therefore $i\chi^{(f)}=i (\alpha\gamma) = (i\alpha)\gamma=j\gamma$, which gives $j\in \left(i\chi^{(f)}\right)\gamma^{-1}$. Moreover, by Theorem \ref{d-green}$\rm(ii)$ we deduce that $\pi_{X_i}(f)\varphi\subseteq \pi_{X_j}(g)$. Next by Theorem \ref{d-green}$\rm(i)$, we have $\gamma=\beta\chi^{(f)}$. A similar argument using $\gamma=\beta\chi^{(f)}$ establishes that $k\in (i\gamma)\left(\chi^{(f)}\right)^{-1}$. Then by Theorem \ref{d-green}(iii), we deduce that $\pi_{X_i}(g)\varphi^{-1}\subseteq \pi_{X_k}(f)$. Thus \rm(ii) holds.
	
\vspace{0.1cm}

	Conversely, assume that the given conditions hold. Note from \rm(i) that  $\left(\chi^{(f)},\gamma\right)\in \mathcal{L}$. For each $i\in I$, by \rm(ii) we fix $j_i\in \left(i\chi^{(f)}\right)\gamma^{-1}$ and $k_i\in (i\gamma)\left(\chi^{(f)}\right)^{-1}$ such that $\pi_{X_i}(f)\varphi\subseteq \pi_{X_{j_i}}(g)$ and $\pi_{X_i}(g)\varphi^{-1}\subseteq \pi_{X_{k_i}}(f)$. Define $\alpha,\beta\in T(I)$ given by $i\alpha=j_i$ and $i\beta=k_i$. Now we show that $\chi^{(f)}=\alpha\gamma$. For this, let $i\in I$. Since $j_i\in \left(i\chi^{(f)}\right)\gamma^{-1}$, we obtain $i\chi^{(f)} = j_i\gamma = (i\alpha)\gamma = i(\alpha\gamma)$, which yields $\chi^{(f)}=\alpha\gamma$. We can similarly prove that $\gamma=\beta\chi^{(f)}$, since $k_i\in (i\gamma)\left(\chi^{(f)}\right)^{-1}$. By \rm(i), we have $(\gamma,\chi^{(g)})\in \mathcal{R}$ in $T(I)$. Thus Theorem \ref{d-green}\rm(i) holds. By the definition of $\alpha$ and $\beta$, we also deduce that $\pi_{X_i}(f)\varphi\subseteq \pi_{X_{i\alpha}}(g)$ for all $i\in I$, and $\pi_{X_i}(g)\varphi^{-1}\subseteq \pi_{X_{i\beta}}(f)$ for all $i\in I$. Thus each condition of Theorem \ref{d-green} holds. Hence, since $T_{T(I)}(X,\mathcal{P})=T(X,\mathcal{P})$, we conclude from Theorem \ref{d-green} that $(f,g)\in \mathcal{D}$ in $T(X,\mathcal{P})$.
\end{proof}



For every $x\in X$, denote by $\bar{x}$ the block of $\mathcal{P}$ that contains $x$. Thus $\bar{x}=X_i$ for some $i\in I$.

\vspace{0.05cm}
We require the following result to characterize the relation $\mathcal{J}$ on $T_{\mathbb{S}(I)}(X,\mathcal{P})$.

\begin{lemma}\label{j-lemma}
Let $f,g\in T_{\mathbb{S}(I)}(X,\mathcal{P})$. Then $J_f\leq J_g$ in $T_{\mathbb{S}(I)}(X,\mathcal{P})$ if and only if there exist $\alpha,\beta\in \mathbb{S}(I)$ and a mapping $\varphi\colon Xg\to X$ such that:
\begin{enumerate}
\item[\rm(i)] $X_if\subseteq (X_{i\alpha}g)\varphi$ for all $i\in I$;
\item[\rm(ii)] $(X_i\cap Xg)\varphi\subseteq X_{i\beta}$ for all $i\in I\chi^{(g)}$.
\end{enumerate}
\end{lemma}

\begin{proof}[\textbf{Proof}]
Assume that $J_f\leq J_g$ in $T_{\mathbb{S}(I)}(X,\mathcal{P})$. Then $f=h_1gh_2$ for some $h_1,h_2\in T_{\mathbb{S}(I)}(X,\mathcal{P})$. This gives $\chi^{(f)}=\chi^{(h_1gh_2)}=\chi^{(h_1)}\chi^{(g)}\chi^{(h_2)}$ by Lemma \ref{rak-lemma-2.3}. Write $\chi^{(h_1)}=\alpha$ and $\chi^{(h_2)}=\beta$. 

\vspace{0.05cm}
If $Xg\setminus X(h_1g)\neq \varnothing$, then choose $y\in \bar{x}\cap X(h_1g)$ for each $x\in Xg\setminus X(h_1g)$ with $\bar{x}\cap X(h_1g)\neq \varnothing$; choose $z\in X_{i\beta}$ for each $x\in Xg\setminus X(h_1g)$ with $\bar{x}\cap X(h_1g)=\varnothing$, where $\bar{x}=X_i$ for some $i\in I$. Define a mapping $\varphi \colon Xg\to X$ by:
\begin{eqnarray*}
x\varphi=
\begin{cases}
	xh_2  &\text{if $x\in X(h_1g)$};\\
	yh_2 &\text{if $x\in Xg\setminus X(h_1g)$ and $\bar{x}\cap X(h_1g)\neq\varnothing$};\\
	z   &\text{if $x\in Xg\setminus X(h_1g)$ and $\bar{x}\cap X(h_1g)=\varnothing$}.
\end{cases}		
\end{eqnarray*}
To prove \rm(i), let $i\in I$. Note that $X_i(h_1g)\subseteq X(h_1g)$, and so $x\varphi=xh_2$ for all $x\in X_i(h_1g)$ by the definition of $\varphi$. Therefore, since $f=h_1gh_2$, we obtain
\[X_if=X_i(h_1gh_2)=(X_i(h_1g))h_2=(X_i(h_1g))\varphi =((X_ih_1)g)\varphi\subseteq (X_{i\alpha}g)\varphi.\]
Thus $\rm(i)$ holds. To prove \rm(ii), let $i\in I\chi^{(g)}$ and $x\in X_i\cap Xg$. Then there are two cases:

\vspace{0.02cm}
\noindent \textbf{Case 1:} Suppose $x\in X_i\cap X(h_1g)$. Then $x\in X(h_1g)$. Therefore, since $\chi^{(h_2)}=\beta$, we obtain $x\varphi=xh_2\in X_{i\beta}$.

\vspace{0.02cm}
\noindent \textbf{Case 2:} Suppose $x\in Xg\setminus X(h_1g)$. Then $x\varphi=yh_2 \in X_{i\beta}$ or $x\varphi=z\in X_{i\beta}\in X_{i\beta}$.

\vspace{0.05cm}
Thus, in either case, we have $x\varphi\in X_{i\beta}$. Therefore $(X_i\cap Xg)\varphi\subseteq X_{i\beta}$, and thus $\rm(ii)$ holds.
	
\vspace{0.2mm}
Conversely, assume that the given conditions hold. To prove $J_f\leq J_g$ in $T_{\mathbb{S}(I)}(X,\mathcal{P})$, we will construct $h_1, h_2 \in T_{\mathbb{S}(I)}(X,\mathcal{P})$ such that $f=h_1gh_2$. To define $h_1\in T(X)$, let $i\in I$. For each $x\in X_i$, by $\rm(i)$ we fix $y\in X_{i\alpha}$ such that $xf=y(g\varphi)$. Define $h_1\in T(X)$ by $xh_1=y$ whenever $x\in X_i$. It is easy to verify that $\chi^{(h_1)}=\alpha$, and so $h_1\in T_{\mathbb{S}(I)}(X,\mathcal{P})$.
	
\vspace{1mm}
Next for each $i\in I$, we fix $x_i\in X_i$. Define $h_2\in T(X)$ by:
\begin{eqnarray*}
xh_2=
\begin{cases}
	x\varphi      & \text{if $x\in X_i\cap Xg$, where $i\in I\chi^{(g)}$};\\
	x_{i\beta} & \text{if $x\in X_i\setminus Xg$, where $i\in I\chi^{(g)}$};\\
	x_{j\beta} & \text{if $x\in X_j$, where $j\in I\setminus I\chi^{(g)}$}.
\end{cases}		
\end{eqnarray*}
By \rm(ii) and the definition of $h_2$, it is easy to see that $\chi^{(h_2)}=\beta$, and so $h_2\in T_{\mathbb{S}(I)}(X,\mathcal{P})$. To prove $f=h_1gh_2$, let $x\in X$. Then $x\in X_i$ for some $i\in I$. Note that $xh_1=y$ and $xf=y(g\varphi)$. Therefore $x(h_1gh_2)=(xh_1)gh_2=(yg)h_2 = (yg)\varphi=xf$, which yields $f=h_1gh_2$. Hence $J_f\leq J_g$ in $T_{\mathbb{S}(I)}(X,\mathcal{P})$.
\end{proof}		

\begin{theorem}\label{j-green}
Let $f,g\in T_{\mathbb{S}(I)}(X,\mathcal{P})$. Then $(f,g)\in \mathcal{J}$ in $T_{\mathbb{S}(I)}(X,\mathcal{P})$ if and only if there exist $\alpha,\beta,\gamma,\delta\in \mathbb{S}(I)$ and mappings $\varphi\colon Xg\to X$, $\psi\colon Xf\to X$ such that:
\begin{enumerate}

\item[\rm(i)] $X_if\subseteq (X_{i\alpha}g)\varphi$ and $X_ig\subseteq (X_{i\gamma}f)\psi$  for all $i\in I$;
	
	\item[\rm(ii)] $(X_i\cap Xg)\varphi\subseteq X_{i\beta}$ for all $i\in I\chi^{(g)}$, and $(X_j\cap Xf)\psi\subseteq X_{j\delta}$ for all $j\in I\chi^{(f)}$.
\end{enumerate}
\end{theorem}

\begin{proof}[\textbf{Proof}]
It follows directly from Lemma \ref{j-lemma}.	
\end{proof}


Let $Y$ and $Z$ be nonempty subsets of $X$, and let $\varphi\colon Y\to Z$ be a mapping. We say that $\varphi$ is \textit{$E$-preserving} if for all $x,y\in Y$, $(x, y)\in E$ implies that $(x\varphi,y\varphi)\in E$ (cf. \cite[p. 112]{pei-ca05}). Equivalently, the mapping $\varphi$ is $E$-preserving if and only if for each $X_i\in \mathcal{P}$, there exists $X_j\in \mathcal{P}$ such that $(X_i\cap Y)\varphi\subseteq X_j\cap Z$.

\vspace{1.0mm}
In 2005, Pei \cite[Theorem 3.8]{pei-ca05} determined the relation $\mathcal{J}$ on $T(X,\mathcal{P})$ for a specific case. After that in 2007, Pei et al. \cite[Theorem 3.8]{pei-ca07} gave a complete description of the relation $\mathcal{J}$ on $T(X,\mathcal{P})$. Sangkhanan and Sanwong \cite[Corollary 3.16]{sanwong-sf20} also obtained the same description of the relation $\mathcal{J}$ on $T(X,\mathcal{P})$. Using Theorem \ref{j-green}, we find a more concrete characterization of the relation $\mathcal{J}$ on $T(X,\mathcal{P})$ as follows.


\begin{proposition}
Let $f,g\in T(X,\mathcal{P})$. Then $(f,g)\in \mathcal{J}$ in $T(X,\mathcal{P})$ if and only if there exist $E$-preserving mappings $\varphi\colon Xg\to X$ and $\psi\colon Xf\to X$ such that for each $i\in I$, there exist $j,k\in I$ such that $X_if\subseteq (X_jg)\varphi$ and $X_ig\subseteq (X_kf)\psi$.
\end{proposition}

\begin{proof}[\textbf{Proof}]
Assume that $(f,g)\in \mathcal{J}$ in $T(X,\mathcal{P})$. By Theorem \ref{j-green}, since $T(X,\mathcal{P})=T_{T(I)}(X,\mathcal{P})$, there exist $\alpha,\beta,\gamma,\delta\in T(I)$ and mappings $\varphi\colon Xg\to X$, $\psi\colon Xf\to X$ such that each condition of Theorem \ref{j-green} hold. By Theorem \ref{j-green}\rm(ii), it is clear that both mappings $\varphi$ and $\psi$ are $E$-preserving. To prove the remaining parts, let $i\in I$. Write $i\alpha=j$ and $i\beta=k$. By Theorem \ref{j-green}\rm(i), we then deduce that $X_if\subseteq (X_jg)\varphi$ and $X_ig\subseteq (X_kf)\psi$.

\vspace{0.2mm}
Conversely, assume that the given conditions hold. For each $i\in I$, by hypothesis we choose $j_i\in I$ such that $X_if\subseteq (X_{j_i}g)\varphi$. Define $\alpha\in T(I)$ by $i\alpha=j_i$. Then by hypothesis, we have $X_if\subseteq (X_{i\alpha}g)\varphi$ for all $i\in I$.

\vspace{0.2mm}
Note by hypothesis that the mapping $\varphi$ is $E$-preserving. Therefore for each $i\in I\chi^{(g)}$, there exists unique $i'\in I$ such that $(X_i\cap Xg)\varphi\subseteq X_{i'}$. Define $\beta\in T(I)$ by:
\begin{eqnarray*}
i\beta=
\begin{cases}
	i'   &  \text{if $i\in I\chi^{(g)}$};\\
	i    &  \text{otherwise}.
\end{cases}
\end{eqnarray*}
It is clear from the definition of $\beta$ that $(X_i\cap Xg)\varphi\subseteq X_{i\beta}$ for all $i\in I\chi^{(g)}$. To prove $\chi^{(f)}=\alpha\chi^{(g)}\beta$, let $i\in I$. Write $i\alpha=j$ and $j\chi^{(g)}=k$. Since $k\in I\chi^{(g)}$, we get $k\beta=k'$. It follows that $(X_k\cap Xg)\varphi\subseteq X_{k'}$ by the definition of $\beta$. Since $j\chi^{(g)}=k$, we obtain
\[ X_if\subseteq (X_{i\alpha}g)\varphi=(X_jg)\varphi\subseteq (X_k\cap Xg)\varphi \subseteq X_{k'},\]
which gives $i\chi^{(f)}=k'$. Then $i(\alpha\chi^{(g)}\beta)= (i\alpha)\chi^{(g)}\beta= \left(j\chi^{(g)}\right)\beta=k\beta=k'=i\chi^{(f)}$, and therefore $\chi^{(f)}=\alpha\chi^{(g)}\beta$.

\vspace{0.2mm}
Again by hypothesis, for each $i\in I$ we choose $k_i\in I$ such that $X_ig\subseteq (X_{k_i}f)\psi$. Define $\gamma\in T(I)$ by $i\gamma=k_i$. Then by hypothesis, we have $X_ig\subseteq (X_{i\gamma}f)\psi$ for all $i\in I$.

\vspace{0.2mm}
Finally by hypothesis, the mapping $\psi$ is $E$-preserving. Therefore for each $j\in I\chi^{(f)}$, there exists unique $j'\in I$ such that $(X_j\cap Xf)\psi\subseteq X_{j'}$. Define $\delta\in T(I)$ by:
\begin{eqnarray*}
j\delta=
\begin{cases}
	j'   &  \text{if $j\in I\chi^{(f)}$};\\
	j    &  \text{otherwise}.
\end{cases}
\end{eqnarray*}
It is clear from the definition of $\delta$ that $(X_i\cap Xf)\psi\subseteq X_{i\delta}$ for all $i\in I\chi^{(f)}$. To prove $\chi^{(g)}=\gamma\chi^{(f)}\delta$, let $i\in I$. Write $i\gamma=j$ and $j\chi^{(f)}=k$. Since $k\in I\chi^{(f)}$, we get $k\delta=k'$. It follows that $(X_k\cap Xf)\psi\subseteq X_{k'}$ by the definition of $\delta$. Since $j\chi^{(f)}=k$, we obtain
\[X_ig\subseteq (X_{i\gamma}f)\psi=(X_jf)\psi\subseteq (X_k\cap Xf)\psi \subseteq X_{k'},\]
which gives $i\chi^{(g)}=k'$. Then $i\left(\gamma\chi^{(f)}\delta\right)= (i\gamma)\chi^{(f)}\delta= \left(j\chi^{(f)}\right)\delta=k\delta=k'=i\chi^{(g)}$, and therefore $\chi^{(g)}=\gamma\chi^{(f)}\delta$.

\vspace{0.2mm}
Thus, since $T_{T(I)}(X,\mathcal{P})=T(X,\mathcal{P})$, we conclude from Theorem \ref{j-green} that $(f,g)\in \mathcal{J}$ in $T(X,\mathcal{P})$.
\end{proof}	


\end{document}